\documentclass[final,11pt,a4paper]{amsart} 

\usepackage{tikz,tikz-cd}
\usetikzlibrary{snakes}
\usetikzlibrary{shapes.geometric,positioning}
\usetikzlibrary{arrows,decorations.pathmorphing,decorations.pathreplacing}

\usepackage{graphicx}
\usepackage[all]{xy}
\usepackage{placeins}
\usepackage{enumitem}
\usepackage{amssymb}
\usepackage{latexsym}
\usepackage{amsmath}
\usepackage{mathrsfs}
\usepackage{array,booktabs}
\usepackage{verbatim}
\usepackage{fullpage}
\usepackage[notref, notcite]{showkeys}  
\usepackage{xcolor}
\usepackage{float}
\usepackage{comment}
\usepackage[normalem]{ulem}
\usepackage{dsfont}
\usepackage[T1]{fontenc}

\usepackage{hyperref}
\hypersetup{
    colorlinks,
    linkcolor={red!50!black},
    citecolor={blue!50!black},
    urlcolor={blue!80!black}
}

\newcommand{\harxiv}[1]{ \href{http://arxiv.org/abs/#1}{\texttt{arXiv:#1}}}
\newcommand{\hyref}[2]{ \hyperref[#2]{#1~\ref*{#2}} }

\newcommand{\coloneqq}{\mathrel{\mathop:}=}

\newcommand{\Canakci}{\c{C}anak\c{c}\i}
\renewcommand{\comment}[1]{{}}

\theoremstyle{plain}
\newtheorem{theorem}{Theorem}[section]

\newtheorem{lemma}[theorem]{Lemma}
\newtheorem{corollary}[theorem]{Corollary}
\newtheorem{proposition}[theorem]{Proposition}

\theoremstyle{definition}
\newtheorem{remark}[theorem]{Remark}
\newtheorem{example}[theorem]{Example}
\newtheorem{definition}[theorem]{Definition}

\newcommand{\sD}{\mathsf{D}}
\newcommand{\sK}{\mathsf{K}}

\DeclareMathAlphabet{\mathpzc}{OT1}{pzc}{m}{it}

\newcommand{\bC}{\mathbb{C}}

\newcommand{\inv}[1]{\overline{#1}}
\newcommand{\diaginv}[1]{\overline{#1}}

\newcommand{\Db}{\sD^b}
\newcommand{\KminusL}{\sK^{b,-}(\proj{\Lambda})}

\newcommand{\kk}{{\mathbf{k}}}

\renewcommand{\emptyset}{\varnothing}

\DeclareMathOperator{\Hom}{\mathrm{Hom}}

\newcommand{\proj}[1]{\mathsf{proj}(#1)}

\newcommand{\too}{\longrightarrow}
\newcommand{\rightlabel}[1]{\stackrel{#1}{\longrightarrow}}

\newcommand{\length}[1]{\mathrm{len}(#1)}

\newcommand{\B}{B^{\bullet}}
\renewcommand{\P}{P^{\bullet}}
\newcommand{\Q}{Q^{\bullet}}
\newcommand{\M}{M^{\bullet}}
\newcommand{\N}{N^{\bullet}}

\newcommand{\f}{f^{\bullet}}

\renewcommand{\i}{i^{\bullet}}

\newcommand{\xydot}{{\bullet}}
\newcommand{\arr}{\ar@{-}[r]}

\newenvironment{pmat}{\left[ \begin{smallmatrix}}{\end{smallmatrix} \right]}

\renewcommand{\phi}{\varphi}
\renewcommand{\epsilon}{\varepsilon}

\begin{document}

\title{Addendum and Erratum: Mapping cones for morphisms involving a band complex in the bounded derived category of a gentle algebra}

\author{\.{I}lke \Canakci} 
\address{Department of Mathematics, VU, Amsterdam, Amsterdam, 1081 HV, The Netherlands.}
\email{i.canakci@vu.nl}

\author{David Pauksztello} 
\address{Department of Mathematics and Statistics, Lancaster University, Lancaster, LA1 4YF, United Kingdom}
\email{d.pauksztello@lancaster.ac.uk}

\author{Sibylle Schroll}
\address{Department of Mathematics, University of Leicester, University Road, Leicester, LE1 7RH, United Kingdom.}
\email{schroll@leicester.ac.uk}

\keywords{bounded derived category, gentle algebra, homotopy string and band, string combinatorics, mapping cone
}
\thanks{This work has been supported through the EPSRC grant EP/P016014/1 for the first named author and the EPSRC Early Career Fellowship EP/P016294/1 for the third named author.}

\subjclass[2010]{18E30, 16G10, 05E10}

\begin{abstract}
In this note we correct two oversights in [{\it Mapping cones in the bounded derived category of a gentle algebra}, J. Algebra {\bf 530} (2019), 163--194] which only occur when a band complex is involved. 
As a consequence we see that the mapping cone of a morphism between two band complexes can decompose into arbitrarily many indecomposable direct summands.
\end{abstract}

\maketitle


Let $\Lambda$ be a gentle algebra and consider indecomposable complexes $\P$ and $\Q$ in the bounded derived category $\Db(\Lambda)$. In \cite{ALP}, a canonical basis for $\Hom_{\Db(\Lambda)}(\P,\Q)$ is given in terms of homotopy string and band combinatorics \cite{BM,Bo,BD}.
Given an element of the canonical basis $\f \in \Hom_{\Db(\Lambda)}(\P,\Q)$, in \cite{CPS1}, we compute the decomposition of its mapping cone $\M_{\f}$ into indecomposable string and band complexes. This is then applied in \cite{CPS2} to give an explicit description of the Ext spaces between indecomposable modules over a gentle algebra.
Furthermore, the computation in \cite{CPS1} has been applied in \cite{OPS} (see also \cite{LP}) to give an interpretation of the mapping cone calculus in the context of Fukaya categories of surfaces.

Recall from \cite{ALP,CPS1} that if $\f$ is either a graph map or an element of the homotopy class determined by a quasi-graph map, then $\f$ is determined by an overlap in the homotopy string or bands corresponding to $\P$ and $\Q$; see Section~\ref{sec:finite} for precise details.
Unfortunately, in \cite{CPS1}, when at least one of $\P$ or $\Q$ is a band complex we made the following oversights.
\begin{itemize}
\item The case analysis for elements of the canonical basis corresponding to (quasi-)graph maps was not complete. Namely, we did not consider the cases of (quasi-)graph maps in which the overlap is longer than at least one homotopy band.
\item In the case that both $\P$ and $\Q$ are band complexes, in the description of $\M_{\f}$, the resulting combinatorial word may be a nontrivial power of a homotopy band. (In \cite{CPS1}, we erroneously claimed the resulting word was a homotopy band.)
\end{itemize}
In this note, we rectify these oversights. 
A remarkable consequence of these corrections is that typically the mapping cone of a morphism between two band complexes decomposes into more than one indecomposable direct summand.

Let us briefly outline the content of the note. In Section~\ref{sec:finite} we discuss the cases omitted from consideration in \cite{CPS1} and set up some notation and terminology to treat them. In Section~\ref{sec:graph-maps} we describe the mapping cones of graph maps in which at least one indecomposable complex is a band complex, correcting the statements of \cite[Prop. 2.9, 2.11 \& 2.12]{CPS1}. A key observation here is Lemma~\ref{lem:power-of-band}, which explains how to modify the word combinatorics in the case that the resulting word is a nontrivial power of a band. For this lemma, one needs the hypothesis that the ground field is algebraically closed. In Section~\ref{sec:quasi}, we look at quasi-graph maps and extend and correct the statement of \cite[Prop. 5.2]{CPS1}. Finally, in Section~\ref{sec:singleton}, in light of Lemma~\ref{lem:power-of-band}, we correct the descriptions of the mapping cones of single and double maps involving two band complexes, \cite[Prop. 3.4 \& 4.2]{CPS1}.
Throughout this note we shall use the notation and terminology from \cite{CPS1} without further mention.

\subsection*{Acknowledgment}

The authors are grateful to Rosanna Laking whose question on \cite{CPS2} highlighted the cases which were not considered in \cite{CPS1}.
The second named author would also like to thank Raquel Coelho Sim\~oes for useful conversations about this note and Karin Baur and Raquel Coelho Sim\~oes for sharing a preliminary version of \cite{BCS-addendum}.

\section{Finite versus infinite} \label{sec:finite}

Throughout this section, $\Lambda$ will be a gentle algebra over $\kk$. We start by defining the length of a (finite) homotopy string or homotopy band.

\begin{definition} \label{def:length}
Let $\sigma = \sigma_n \cdots \sigma_2 \sigma_1$ be a homotopy letter partition of a finite homotopy string or homotopy band. The \emph{length of $\sigma$} is $\length{\sigma} \coloneqq n$.
Note that if $\sigma$ is a homotopy band then $\length{\sigma}$ is even and at least $2$.
\end{definition}

Throughout this note we shall be considering graph maps and quasi-graph maps between indecomposable complexes in the bounded derived category such that at least one of those complexes is a band complex. 
Let $\sigma$ and $\tau$ be homotopy strings or bands. Recall from \cite[\S 1.4.1]{CPS1} or \cite[\S 3.2]{ALP},
\begin{itemize}
\item if $\f \colon \Q_\sigma \to \Q_\tau$ is a graph map then $\f$ is determined by a (possibly trivial) maximal common homotopy substring $\rho$ subject to certain endpoint conditions; and,
\item if $\phi \colon \Q_\sigma \rightsquigarrow \Sigma^{-1} \Q_\tau$ is a quasi-graph map then $\phi$ is also determined by a (possibly trivial) maximal common homotopy substring $\rho$ subject to another set of endpoint conditions.
\end{itemize}
In both cases, we shall refer to the maximal common homotopy substring $\rho$ determining the (quasi-)graph map as the \emph{overlap}. The \emph{length of the overlap} $\rho$ is defined as in Definition~\ref{def:length} above.
If $\sigma$ and $\tau$ are either infinite homotopy strings or homotopy bands, it is possible for the overlap to be infinite.

\subsection{The unfolded diagram of a homotopy band} \label{sec:unfolded}

For this section, we refer to \cite[\S 2.2]{ALP}. Let $\lambda \in \kk^*$ and suppose $(\sigma,\lambda)$ is a homotopy band. Write $\sigma = \sigma_n \cdots \sigma_1$, where we assume, for now, without loss of generality that $\sigma_1$ is a direct homotopy letter. Then the \emph{unfolded diagram of a homotopy band} is an infinite repeating diagram,
\[
\xymatrix{ \cdots \arr^-{\sigma_2} & \xydot \ar[r]^-{\lambda \sigma_1} & \xydot \arr^-{\sigma_n} & \xydot \ar@{..}[r] & \xydot \ar[r]^-{\lambda \sigma_1} & \xydot \arr^-{\sigma_n} & \xydot \arr^-{\sigma_{n-1}} & \cdots}.
\]
By convention in the following, we shall always assume that the scalar $\lambda$ is placed on a direct homotopy letter. Note that the number of direct homotopy letters in a homotopy band is precisely half the number of homotopy letters.

Given this description of the unfolded diagram, for a homotopy band $\sigma$ we write ${}^\infty \sigma^\infty$ for the word formed by infinitely many concatenations of the homotopy band $\sigma$. In particular, overlaps $\rho$ determining (quasi-)graph maps involving a band complex $\B_{\sigma,\lambda}$ are therefore homotopy substrings of the infinite word ${}^\infty \sigma^\infty$. In particular, one now immediately sees how overlaps determining (quasi-)graph maps which are `longer' than a homotopy band occur. 
Below we give an example illustrating both the graph map and quasi-graph map situation.

\begin{example} \label{ex:two-kroneckers}
Let $\Lambda$ be given by the following quiver with relations.
\[
\begin{tikzpicture}[scale=1.2]
\node (A) at (0,0){$3$};
\node (B) at (1,0){$1$};
\node (C) at (2,0){$2$};
\draw[transform canvas={yshift=-2.8pt},->] (A) -- node[below,scale=.8]{$d$} (B);
\draw[transform canvas={yshift=2.8pt},->] (A) --  node[above,scale=.8]{$c$} (B);
\draw[transform canvas={yshift=-2.8pt},->] (B) -- node[below,scale=.8]{$b$} (C);
\draw[transform canvas={yshift=2.8pt},->] (B) -- node[above,scale=.8]{$a$} (C);
\draw[thick,dotted] (1.25,.1) arc (20:160:.32cm)
(1.25,-.1) arc (-20:-160:.32cm) 
;
\end{tikzpicture}
\]
Let $\sigma = d \inv{c}  \, \inv{a} b (d \inv{c})^2 \inv{a} b$ 
and $\tau = (d \inv{c})^2 \inv{a} b$. 
The following unfolded diagram determines a quasi-graph map $\B_{\sigma,1} \rightsquigarrow \B_{\tau,1}$,
\begin{equation} \label{unfolded}
\begin{tikzpicture}
\node[scale=.6] at (0,0){$
\begin{tikzpicture}
\matrix (m) [matrix of math nodes,row sep=2em,column sep=2.4em,minimum width=2em]
  {
   2 & 1 & 3 & 1 & 2 & 1 & 3 & 1 & 3 & 1 & 2 & 1 & 3 & 1 & 2 & 1 \\
   3 & 1 & 3 & 1 & 2 & 1 & 3 & 1 & 3 & 1 & 2 & 1 & 3 & 1 & 3 & 1\\};
     
\draw[->] (m-1-1)-- node [midway,anchor=south west,scale=.9]{$b$} (m-1-2);
\draw[<-] (m-2-1)-- node [anchor=north west,scale=.9]{$\bar{c}$} (m-2-2);     

\draw[->]  (m-1-2)-- node [above,scale=.9]{$d$} (m-1-3);
\draw[->]  (m-2-2)-- node [below,scale=.9]{$d$} (m-2-3);

\draw[<-] (m-1-3)-- node [above,scale=.9]{$\bar{c}$} (m-1-4);
\draw[<-] (m-2-3)-- node [below,scale=.9]{$\bar{c}$} (m-2-4);

\draw[<-] (m-1-4)-- node [above,scale=.9]{$\bar{a}$} (m-1-5);
\draw[<-] (m-2-4)-- node [below,scale=.9]{$\bar{a}$} (m-2-5);

\draw[->] (m-1-5)-- node [above,scale=.9]{$b$} (m-1-6);
\draw[->] (m-2-5)-- node [below,scale=.9]{$b$} (m-2-6);

\draw[->] (m-1-6)-- node [above,scale=.9]{$d$} (m-1-7);
\draw[->] (m-2-6)-- node [below,scale=.9]{$d$} (m-2-7);

\draw[<-] (m-1-7)-- node [above,scale=.9]{$\bar{c}$} (m-1-8);
\draw[<-] (m-2-7)-- node [below,scale=.9]{$\bar{c}$} (m-2-8);

\draw[->] (m-1-8)-- node [above,scale=.9]{$d$} (m-1-9);
\draw[->] (m-2-8)-- node [below,scale=.9]{$d$} (m-2-9);

\draw[<-] (m-1-9)-- node [above,scale=.9]{$\bar{c}$} (m-1-10);
\draw[<-] (m-2-9)-- node [below,scale=.9]{$\bar{c}$} (m-2-10);

\draw[<-] (m-1-10)-- node [above,scale=.9]{$\bar{a}$} (m-1-11);
\draw[<-] (m-2-10)-- node [below,scale=.9]{$\bar{a}$} (m-2-11);

\draw[->] (m-1-11)-- node [above,scale=.9]{$b$} (m-1-12);
\draw[->] (m-2-11)-- node [below,scale=.9]{$b$} (m-2-12);

\draw[->] (m-1-12)-- node [above,scale=.9]{$d$} (m-1-13);
\draw[->] (m-2-12)-- node [below,scale=.9]{$d$} (m-2-13);

\draw[<-] (m-1-13)-- node [above,scale=.9]{$\bar{c}$} (m-1-14);
\draw[<-] (m-2-13)-- node [below,scale=.9]{$\bar{c}$} (m-2-14);

\draw[<-] (m-1-14)-- node [above,scale=.9]{$\bar{a}$} (m-1-15);
\draw[->] (m-2-14)-- node [below,scale=.9]{$d$} (m-2-15);

\draw[->] (m-1-15)-- node [above,scale=.9]{$b$} (m-1-16);
\draw[<-] (m-2-15)-- node [below,scale=.9]{$\bar{c}$} (m-2-16);


\draw[transform canvas={xshift=-1.5pt}] (m-1-2) -- node [anchor=south east,scale=.9]{} (m-2-2);
\draw[transform canvas={xshift=1.5pt}] (m-1-2) --  (m-2-2);

\draw[transform canvas={xshift=-1.5pt}] (m-1-3) -- (m-2-3);
\draw[transform canvas={xshift=1.5pt}] (m-1-3) --  (m-2-3);

\draw[transform canvas={xshift=-1.5pt}] (m-1-4) -- (m-2-4);
\draw[transform canvas={xshift=1.5pt}] (m-1-4) --  (m-2-4);

\draw[transform canvas={xshift=-1.5pt}] (m-1-5) -- (m-2-5);
\draw[transform canvas={xshift=1.5pt}] (m-1-5) --  (m-2-5);

\draw[transform canvas={xshift=-1.5pt}] (m-1-6) -- (m-2-6);
\draw[transform canvas={xshift=1.5pt}] (m-1-6) --  (m-2-6);

\draw[transform canvas={xshift=-1.5pt}] (m-1-7) -- (m-2-7);
\draw[transform canvas={xshift=1.5pt}] (m-1-7) --  (m-2-7);

\draw[transform canvas={xshift=-1.5pt}] (m-1-8) -- (m-2-8);
\draw[transform canvas={xshift=1.5pt}] (m-1-8) --  (m-2-8);

\draw[transform canvas={xshift=-1.5pt}] (m-1-9) -- (m-2-9);
\draw[transform canvas={xshift=1.5pt}] (m-1-9) --  (m-2-9);

\draw[transform canvas={xshift=-1.5pt}] (m-1-10) -- (m-2-10);
\draw[transform canvas={xshift=1.5pt}] (m-1-10) --  (m-2-10);

\draw[transform canvas={xshift=-1.5pt}] (m-1-11) -- (m-2-11);
\draw[transform canvas={xshift=1.5pt}] (m-1-11) --  (m-2-11);

\draw[transform canvas={xshift=-1.5pt}] (m-1-12) -- (m-2-12);
\draw[transform canvas={xshift=1.5pt}] (m-1-12) --  (m-2-12);

\draw[transform canvas={xshift=-1.5pt}] (m-1-13) -- (m-2-13);
\draw[transform canvas={xshift=1.5pt}] (m-1-13) --  (m-2-13);

\draw[transform canvas={xshift=-1.5pt}] (m-1-14) -- (m-2-14);
\draw[transform canvas={xshift=1.5pt}] (m-1-14) --  (m-2-14);

%
\node (A) [xshift=2em,left of=m-1-2, yshift=-1.4em,above of=m-1-2] {};
\node (B) [xshift=3em,left of=m-1-2, yshift=-1.4em,above of=m-1-12] {};
\node (A') [yshift=.8em,below of=A] {};
\node (B') [yshift=.8em,below of=B] {};

\node (C) [xshift=2em,left of=m-2-2, yshift=-2.1em,above of=m-2-2] {};
\node (D) [xshift=3.4em,left of=m-2-8, yshift=-2.1em,above of=m-2-8] {};
\node (C') [yshift=.75em,below of=C] {};
\node (D') [yshift=.75em,below of=D] {};
\draw[dotted, line width=.7] (A.center) -- (B.center) -- (B'.center) -- (A'.center) --(A.center);
\draw[dotted, line width=.7] (C.center) -- (D.center) -- (D'.center) -- (C'.center) --(C.center)
;
\end{tikzpicture}
$};
\end{tikzpicture}
\end{equation}
where we have indicated in the dotted boxes one copy of each homotopy band.
Reading the diagram `upside down' gives an example of a graph map $\B_{\tau,1} \to \B_{\sigma, 1}$.
In both cases, the overlap $\rho = d\inv{c} \, \inv{a} b (d \inv{c})^2  \inv{a} b d \inv{c}$ is longer than both homotopy bands.
\end{example}

\subsection{Infinite overlaps}

In this section, we observe that we do not need to consider when (quasi-)graph maps are determined by an infinite overlap in our analysis. The case when one of $\sigma$ and $\tau$ is a (possibly infinite) homotopy string and the other is a homotopy band can be eliminated immediately for combinatorial reasons; see \cite[\S 1.4.1 \& \S 1.4.4]{CPS1}.

If both $\sigma$ and $\tau$ are homotopy bands then the consideration of (quasi-)graph maps determined by an infinite overlap is significantly restricted by the following lemma. The statement and proof in \cite{BCS-addendum} is formulated in terms of the combinatorics of (classical/module-theoretic) strings and bands, but the argument can be modified in a straightforward manner to the combinatorics of homotopy strings and homotopy bands.

\begin{lemma}[{\cite[Lem. 1.2]{BCS-addendum}}] \label{lem:infinite}
Let $\Lambda$ be a gentle algebra. Suppose $(\sigma,\lambda)$ and $(\tau,\mu)$ are homotopy bands. If $\f \colon \B_{\sigma,\lambda} \to \B_{\tau,\mu}$ is a graph map or $\phi \colon \B_{\sigma,\lambda} \rightsquigarrow \Sigma^{-1} \B_{\tau,\mu}$ is a quasi-graph map determined by an infinite overlap $\rho$, then, up to suitable rotation and inversion, $\sigma = \tau$. 
\end{lemma}

\begin{corollary} \label{cor:infinite}
Let $\Lambda$ be a gentle algebra. Suppose $(\sigma,\lambda)$ and $(\tau,\mu)$ are homotopy bands. 
\begin{enumerate}
\item If $\f \colon \B_{\sigma,\lambda} \to \B_{\tau,\mu}$ is a graph map determined by an infinite overlap, then $\f$ is an isomorphism. In particular, $\lambda = \mu$.
\item If $\phi \colon \B_{\sigma,\lambda} \rightsquigarrow \Sigma^{-1} \B_{\tau,\mu}$ is a quasi-graph map determined by an infinite overlap, then $\B_{\sigma,\lambda} \cong \Sigma^{-1}\B_{\tau,\mu}$. In particular,  $\lambda = \mu$.
\end{enumerate}
\end{corollary}

Corollary~\ref{cor:infinite} means, therefore, that whenever we have a (quasi-)graph map between band complexes determined by an infinite overlap, either the mapping cone is zero (in the case of a graph map), or else the mapping cone is the middle term of the Auslander--Reiten triangle (in the case of a quasi-graph map). In the second case, this can be seen by observing that under Serre duality the identity map on a band complex is taken to the corresponding quasi-graph map given by the same overlap in the other direction; see \cite[Ex. 5.9]{ALP} and \cite[Rem. 1.8]{CPS2}.

\section{Mapping cones of graph maps involving a band complex} \label{sec:graph-maps}

The key technical tool in the arguments that follow is \cite[Lem. 2.4]{CPS1}. We need a slightly more general version in this note, which for convenience we state here. The proof consists of a straightforward modification of the proof given in \cite{CPS1}. 

\begin{lemma}[{\cite[Lem 2.4]{CPS1}}] \label{lem:homotopy}
Let $\Lambda$ be a finite dimensional $\kk$-algebra, let $\lambda \in \kk^*$ and suppose $\P$ is a complex of the form
\[
\xymatrix{
\P \colon & \cdots \ar[r] & P^{n-1} \ar[r]^-{\begin{pmat} a_1 & a_2 \end{pmat}} & P^n \oplus Q \ar[r]^-{\begin{pmat} b_1 & b _2 \\ b_3 & \lambda \cdot \mathds{1} \end{pmat}} &  P^{n+1} \oplus Q \ar[r]^-{\begin{pmat} c_1 \\ c_2 \end{pmat}} &  P^{n+2} \ar[r] & \cdots,
}
\]
where $\mathds{1}\colon Q \to Q$ denotes the identity morphism on $Q$.
Then $\P \cong (P')^\bullet \oplus \Q$, where $(P')^\bullet$ and $\Q$ are the following complexes,
\begin{align*}
(P')^\bullet \colon & \xymatrix{
\cdots \ar[r] & P^{n-1} \ar[r]^-{a_1} & P^n \ar[rr]^-{b_1 - \lambda^{-1} b_2 b_3} & & P^{n+1} \ar[r]^-{c_1} & P^{n+2} \ar[r] & \cdots
}; \\
\Q \colon & \xymatrix{
\cdots \ar[r] & 0 \ar[r] & Q \ar[r]^-{\lambda \cdot \mathds{1}} & Q \ar[r] & 0 \ar[r] & \cdots
},
\end{align*} 
where $(P')^i = P^i$ and $Q^i = 0$ whenever $i \notin \{n, n+1\}$. In particular, in the homotopy category $\P \cong (P')^\bullet$. 
\end{lemma}

The following theorem extends \cite[Prop. 2.11]{CPS1} to the case of a graph map from a band complex to a string complex in which the overlap determining the graph map is longer than the homotopy band. We see that the effect on the mapping cone of an overlap longer than a homotopy band is to `remove one copy of the homotopy band' from the homotopy string to obtain the homotopy string of the mapping cone. To enable effective comparison with the case that the overlap is at most the length of the homotopy band, we also restate \cite[Prop. 2.11]{CPS1}.  

\begin{theorem} \label{thm:graph-band-to-string}
Let $\Lambda$ be a gentle algebra,  $\sigma$ be a homotopy band  and $\tau$ be a homotopy string. For $\lambda \in \kk^*$,  let $\f \colon \B_{\sigma,\lambda} \to \P_\tau$ be a graph map determined by a (possibly trivial) maximal common homotopy string $\rho = \rho_k \cdots \rho_1$.
\begin{enumerate}
\item \label{graph:original} {\rm (\cite[Prop. 2.11]{CPS1})} Suppose $\rho$ is a proper subword of $\sigma$. Then, after suitable rotation of $\sigma$, there is a decomposition $\sigma = \rho\alpha$ and a decomposition $\tau = \delta \tau_L \rho \tau_R \gamma$. Then $\M_{\f} \cong \P_c$, where $c = \delta \tau_L \inv{\alpha} \tau_R \gamma$. This is indicated in the unfolded diagram below, in which $\alpha = \alpha_\ell \cdots \alpha_1$. Note that, if $f_L \neq \emptyset$ then $\tau_L \inv{\alpha_1} = \inv{f_L}$; similarly, if $f_R \neq \emptyset$ then $\inv{\alpha_\ell} \tau_R = f_R$.
\[
\xymatrix@!R=4px{
\B_\sigma \colon & \xydot \arr^-{\alpha_2} & \xydot \arr^-{\alpha_1} \ar[d]_-{f_L} & \xydot \arr^-{\rho_k} \ar@{=}[d] & \cdots \arr^-{\rho_1}  & \xydot \arr^-{\alpha_\ell} \ar@{=}[d] & \xydot \arr^-{\alpha_{\ell - 1}} \ar[d]^-{f_R} & \xydot \ar@{..}[r] & \xydot \arr^-{\alpha_1} & \xydot \arr^-{\rho_k} & \xydot  \\
\P_\tau \colon      & \ar@{~}[r]_-{\delta} & \xydot \arr_-{\tau_L}                                                & \xydot \arr_-{\rho_k}                  &  \cdots \arr_-{\rho_1}                  & \xydot \arr_-{\tau_R}                                                    & \xydot \ar@{~}[r]_-{\gamma}                 &  & & &
}
\]
\item \label{graph:longer} Suppose there is an integer $\ell \geq 1$ such that $\rho$ is a proper subword of $\sigma^{\ell+1}$ and $\sigma^\ell$ is a (not necessarily proper) subword of $\rho$. Then, after suitable rotation of $\sigma$, there is a decomposition $\sigma = \beta \alpha$, with $\beta$ or $\alpha$ possibly trivial, such that 
\[
\tau = \delta \sigma^\ell \beta \gamma \text{ for some } \ell \geq 1 \text{ and some homotopy strings } \gamma \text{ and } \delta. 
\]
In this case, $\M_{\f} = \P_c$, where $c = \delta \sigma^{\ell - 1} \beta \gamma$.
\end{enumerate}
\end{theorem}

\begin{remark}
Note that in Theorem~\ref{thm:graph-band-to-string}\eqref{graph:longer}, $\gamma$, $\sigma^\ell \beta$ and $\delta$ need not be homotopy substrings in the strictest sense of \cite[Def 1.3(5)]{CPS1}. It is possible when taking a homotopy letter partition (see \cite[Def 1.3(3)]{CPS1}) that the last homotopy letter of $\gamma$ merges with the first homotopy letter of $\sigma^\ell \beta$ and the last homotopy letter of $\sigma^\ell \beta$ merges with the first homotopy letter of $\delta$.
\end{remark}

\begin{proof}[Proof of Theorem~\ref{thm:graph-band-to-string}] 
Statement \eqref{graph:original} was proved in \cite[Prop. 2.11]{CPS1}, therefore, we only prove \eqref{graph:longer}. 
The technique is to use Lemma~\ref{lem:homotopy} repeatedly to remove the identity morphisms occurring in the mapping cone. 
The proof is somewhat technical so we proceed in a sequence of steps.

\smallskip
\noindent \textbf{Step 1:} \textit{Translate Lemma~\ref{lem:homotopy} into the language of unfolded diagrams.}
\smallskip

This step plays the role of \cite[Cor. 2.7]{CPS1} in the proof of \cite[Thm. 2.2]{CPS1}.
Consider the following situation in $\M_{\f}$,
\[
\xymatrix{
\M_{\f} \colon & \cdots \ar[r] & P^{n-1} \ar[r]^-{\begin{pmat} a_1 & a_2 \end{pmat}} & P^n \oplus P_x \ar[r]^-{\begin{pmat} b_1 & b _2 \\ b_3 & 1 \end{pmat}} &  P^{n+1} \oplus P_x \ar[r]^-{\begin{pmat} c_1 \\ c_2 \end{pmat}} &  P^{n+2} \ar[r] & \cdots.
}
\] 
The unfolded diagram has the following form,
\[
\xymatrix@!R=4px{
 x_1 \arr^-{\alpha_1}  & x \ar@{=}[d]  \arr^-{\alpha_2}         & x_2       \ar@{..}[r]  & x_1 \arr^-{\alpha_1}      & x \ar[d]^-{f} \arr^-{\alpha_2} & x_2                        \\
 y_1 \arr_-{\beta_1} & x \arr_-{\beta_2} & y_2 \ar@{..}[r] & z_1 \arr_-{\gamma_1} & x' \arr_-{\gamma_2}          & z_2
} \, ,
\]
where the differentials $d^{n-1}_{\M_{\f}}$, $d^n_{\M_{\f}}$ and $d^{n+1}_{\M_{\f}}$, when they exist, have the following components:
\begin{itemize}
\item the identity map $P_x \to P_x$ on the left hand side corresponds to the component $1$ in the differential $d^n_{\M_{\f}}$;
\item $\alpha_1$ and $\alpha_2$ are either the components of $a_2$ or $b_3$, depending on their orientations;
\item $\beta_1$ and $\beta_2$ are either the components of $b_2$ or $c_2$, depending on their orientations;
\item $\gamma_1$ and $\gamma_2$ are either the components of $b_1$ or $c_1$, depending on their orientations; and
\item $f$ is a component of $b_3$, which may be an identity map, a zero map, or a map induced by a nontrivial path in $(Q,I)$.
\end{itemize}
We note that this is a schematic only. In particular, the right-hand side of the diagram represents many possible repetitions of $P_x$ in the unfolded diagram of $\f$ depending on how many times $P_x$ occurs in the support of $\f$. However, the schematic allows us to describe explicitly the effect of Lemma~\ref{lem:homotopy} on the unfolded diagram of $\M_{\f}$. Namely, we obtain the following new components for $1 \leq i, j \leq 2$:
\begin{itemize}
\item  $-\beta_i \alpha_j \colon P_{y_i} \to P_{x_i}$ is a component of $-b_2 b_3$ whenever $\beta_i$ is a component of $b_2$ and $\alpha_j$ is a component of $b_3$; this is zero otherwise, and
\item $-\beta_i f \colon P_{y_i} \to P_{x'}$ whenever $\beta_i$ is a component of $b_2$; this is zero otherwise.
\end{itemize}
The unfolded diagram of the complex $(\M_{\f})^\prime$ thus becomes the following:
\[
\xymatrix@!R=4px{
 x_1                                &                          & x_2       \ar@{..}[r]  &  x_1                                         &  &   x_2                     \\
 y_1 \ar@{-->}@/^3pc/[urrr]^>{-\beta_1\alpha_1}  \ar@{-->}@/^6pc/[urrrrr]^-{-\beta_1\alpha_2} \ar@{-->}@/_3pc/[rrrr]_-{-\beta_1 f} &   & y_2 \ar@{..}[r] \ar@{-->}[ur]^-{-\beta_2 \alpha_1} \ar@{-->}[urrr]^-{-\beta_2 \alpha_2} \ar@{-->}@/_3pc/[rr]_-{-\beta_2 f} & z_1 \arr_-{\gamma_1} & x' \arr_-{\gamma_2}          & z_2
} \, ,
\]

\smallskip
\noindent \textbf{Step 2.} {\it Remove one copy of $\sigma$ from $\tau$.}
\smallskip

In this step, we iteratively apply Step 1 to remove one copy of $\sigma$ from $\tau$. In this process we will construct either one or two new families of maps in the unfolded diagram, which will be connected to either end of $\delta$ and $\sigma^{\ell-1} \beta \gamma$. Consider the following schematic of the starting unfolded diagram, which shows the leftmost complete overlap with the band together with each occurrence of $x_{n-1}$ in the support of $\f$,
\[
\scalebox{0.85}{
\xymatrix@!R=4px{
x_1 \arr^-{-\sigma_1} & x_0 \arr^-{-\sigma_n} \ar[d]_-{f_L} & x_{n-1} \arr^-{-\sigma_{n-1}} \ar@{=}[d] & x_{n-2} \ar@{..}[r] & x_2 \arr^-{-\sigma_2}  & x_1 \arr^-{-\sigma_1} \ar@{=}[d] & x_0 \arr^-{-\sigma_n} \ar@{=}[d] & x_{n-1} \ar[d]^-{g_k} \ar@{..}[r] & x_0 \arr^-{-\sigma_n} & x_{n-1} \ar[d]^-{g_i} \arr^-{-\sigma_{n-1}} & x_{n-2} \ar@{..}[r] & \\
\ar@{~}[r]_-{\delta} & z \arr_-{\tau_L}                   & x_{n-1} \arr_-{\sigma_{n-1}}           & x_{n-2} \ar@{..}[r] & x_2 \arr_-{\sigma_2}  & x_1 \arr_-{\sigma_1}            & x_0 \arr_-{\omega_k}             & y_k \ar@{..}[r]               & x_{i,0} \arr_-{\omega_i}       & y_i \arr                                 & x_{i,n-2} \ar@{..}[r] &
},}
\] 
where the number of times $x_{n-1}$ occurs in the support of $\f$ is $k+1$, and if $k > 1$ then $y_i = x_{n-1}$, $x_{i,0} = x_0$, $x_{i,n-2} = x_{n-2}$, $g_i = \mathrm{id}_{x_{n-1}}$ and $\omega_i = \sigma_n$ for all $1 <  i \leq k$.
 
We split the analysis up into two cases: $\sigma_n$ is inverse and $\sigma_n$ is direct.

\smallskip
\noindent \textit{Case:} $\sigma_n$ is inverse.
\smallskip

In this case $\tau_L$ may be empty, direct or inverse. In the case that $\tau_L$ is inverse, we must have that $f_L$ is nontrivial. In the other cases, $f_L = \emptyset$.
On the right-hand side, we must have that $\omega_1 \neq \emptyset$ is an inverse homotopy letter and that $g_1$ is a (possibly trivial) path. 
(The fact that $\omega_i$ is an inverse homotopy letter for $1 < i \leq k$ whenever $k > 1$ is implicit in the fact that $\omega_i = \sigma_n$.)

We first treat the case that $\tau_L$ is direct. The unfolded diagram is the following:
\[
\scalebox{0.85}{
\xymatrix@!R=4px{
x_1 \arr^-{-\sigma_1} & x_0 \ar@{<-}[r]^-{-\sigma_n}  & x_{n-1} \arr^-{-\sigma_{n-1}} \ar@{=}[d] & x_{n-2} \ar@{..}[r] & x_2 \arr^-{-\sigma_2}  & x_1 \arr^-{-\sigma_1} \ar@{=}[d] & x_0 \ar@{<-}[r]^-{-\sigma_n} \ar@{=}[d] & x_{n-1} \ar[d]^-{g_k} \ar@{..}[r] & x_0 \ar@{<-}[r]^-{-\sigma_n} & x_{n-1} \ar[d]^-{g_i} \arr^-{-\sigma_{n-1}} & x_{n-2} \ar@{..}[r] & \\
\ar@{~}[r]_-{\delta} & z \ar[r]_-{\tau_L}            & x_{n-1} \arr_-{\sigma_{n-1}}           & x_{n-2} \ar@{..}[r] & x_2 \arr_-{\sigma_2}  & x_1 \arr_-{\sigma_1}            & x_0 \ar@{<-}[r]_-{\omega_k}             & y_k \ar@{..}[r]               & x_{i,0} \ar@{<-}[r]_-{\omega_i}       & y_i \arr                                 & x_{i,n-2} \ar@{..}[r] &
}}
\] 
Application of Step 1 yields the new unfolded diagram: 
\begin{equation} \label{step2-diag2}
\scalebox{0.85}{
\xymatrix@!R=4px{
x_1 \arr^-{-\sigma_1} & x_0                 &   & x_{n-2} \ar@{..}[r] & x_2 \arr^-{-\sigma_2}  & x_1 \arr^-{-\sigma_1} \ar@{=}[d] & x_0 \ar@{=}[d]              & \ar@{..}[r]     & x_0                            &           & x_{n-2} \ar@{..}[r] & \\
\ar@{~}[r]_-{\delta} & z \ar@/_6pc/[rrrrrr]_-{-\tau_L g_k} \ar@/_9pc/[rrrrrrrr]_-{-\tau_L g_i} &    & x_{n-2} \ar@/_3pc/[rrrr]_-{-\sigma_{n-1}g_k} \ar@/_6pc/[rrrrrr]_-{-\sigma_{n-1}g_i} \ar@{..}[r] & x_2 \arr_-{\sigma_2}  & x_1 \arr_-{\sigma_1}            & x_0 \ar@{<-}[r]_-{\omega_k} & y_k \ar@{..}[r] & x_{i,0} \ar@{<-}[r]_-{\omega_i} & y_i \arr  & x_{i,n-2} \ar@{..}[r] &
}}
\end{equation}
where we have ignored all the composite morphisms incident with the top row of the unfolded diagram because they will be removed when $x_0$ and $x_{n-2}$ are removed in subsequent applications of Step 1.
Moreover, note again, that diagram \eqref{step2-diag2} is a schematic: the curved arrows represent two families of morphisms with each family consisting of one arrow for each time $x_{n-1}$ occurs in the support of $\f$. We observe two things, on the bottom row, 
\begin{itemize}
\item any arrow whose target is $x_{n-1}$ is composed with $g_i$ to give a new arrow whose target is $y_i$; and,
\item any arrow whose source is $x_{n-1}$ is removed.
\end{itemize}
Iterating Step 1, we obtain the following unfolded diagram, where all vertices on the top row have now been removed.
\begin{equation} \label{case-1}
\scalebox{1}{
\xymatrix@!R=4px{
\ar@{~}[r]_-{\delta} & z \ar@/_3pc/[rrrrrr]_-{-\tau_L g_k} \ar@/_6pc/[rrrrrrrr]_-{-\tau_L g_i} &    &   &   &  &  & y_k \ar@{..}[r] \ar@/_1pc/[r]_-{-\omega_k g_i} & x_{i,0} \ar@{<-}[r]^-{\omega_i} & y_i \arr  & x_{i,n-2} \ar@{..}[r] &
}}
\end{equation}

The case that $\tau_L = \emptyset$ is a subcase of the above in which $\delta = \emptyset$ and $-\tau_L g_i = \emptyset$ for all $1 \leq i \leq k$. 
We are therefore left with the case that $\tau_L$ is inverse. For this case, the unfolded diagram is the following.
\[
\scalebox{0.85}{
\xymatrix@!R=4px{
x_1 \arr^-{-\sigma_1} & x_0 \ar@{<-}[r]^-{-\sigma_n} \ar[d]_-{f_L} & x_{n-1} \arr^-{-\sigma_{n-1}} \ar@{=}[d] & x_{n-2} \ar@{..}[r] & x_2 \arr^-{-\sigma_2}  & x_1 \arr^-{-\sigma_1} \ar@{=}[d] & x_0 \ar@{<-}[r]^-{-\sigma_n} \ar@{=}[d] & x_{n-1} \ar[d]^-{g_k} \ar@{..}[r] & x_0 \ar@{<-}[r]^-{-\sigma_n} & x_{n-1} \ar[d]^-{g_i} \arr^-{-\sigma_{n-1}} & x_{n-2} \ar@{..}[r] & \\
\ar@{~}[r]_-{\delta} & z \ar@{<-}[r]_-{\tau_L}            & x_{n-1} \arr_-{\sigma_{n-1}}           & x_{n-2} \ar@{..}[r] & x_2 \arr_-{\sigma_2}  & x_1 \arr_-{\sigma_1}            & x_0 \ar@{<-}[r]_-{\omega_k}             & y_k \ar@{..}[r]               & x_{i,0} \ar@{<-}[r]_-{\omega_i}       & y_i \arr                                 & x_{i,n-2} \ar@{..}[r] &
}}
\]
Note that $\inv{\tau_L} = \inv{\sigma_n} f_L$ and $f_L$ is a nontrivial path in $(Q,I)$. After iterating Step 1, the resulting unfolded diagram is given below.
\begin{equation} \label{case-2}
\scalebox{1}{
\xymatrix@!R=4px{
\ar@{~}[r]_-{\delta} & z  &    &   &   &  &  & \ar[llllll]^-{-\inv{f_L}\omega_k} y_k \ar@{..}[r] \ar@/_1pc/[r]_-{-\omega_k g_i} & x_{i,0} \ar@{<-}[r]^-{\omega_i} & y_i \arr  & x_{i,n-2} \ar@{..}[r] &
}}
\end{equation}
We remark that the condition $\inv{\tau_L} = \inv{\sigma_n} f_L \neq 0$ ensures that $\inv{f_L}\omega_k \neq 0$. 

\smallskip
\noindent \textit{Case:} $\sigma_n$ is direct.
\smallskip

In this case $\tau_L \neq \emptyset$ is a direct homotopy letter and $f_L$ is a nontrivial path in $(Q,I)$. 
On the right-hand side, the case distinction is between $\omega_k$ being an empty, inverse or direct homotopy letter. However, $\omega_k$ can only be an empty or inverse homotopy letter when $k=1$, since whenever $k > 1$, we have $\omega_k = \sigma_n$, which is direct.

Suppose $\omega_k$ is direct. Then $\omega_k = \sigma_n g_k$ with $g_k$ a (possibly trivial) path. (Indeed, $g_k$ can only be nontrivial if $k =1$.) The unfolded diagram is the following.
\[
\scalebox{0.85}{
\xymatrix@!R=4px{
x_1 \arr^-{-\sigma_1} & x_0 \ar[r]^-{-\sigma_n} \ar[d]_-{f_L} & x_{n-1} \arr^-{-\sigma_{n-1}} \ar@{=}[d] & x_{n-2} \ar@{..}[r] & x_2 \arr^-{-\sigma_2}  & x_1 \arr^-{-\sigma_1} \ar@{=}[d] & x_0 \ar[r]^-{-\sigma_n} \ar@{=}[d] & x_{n-1} \ar[d]^-{g_k} \ar@{..}[r] & x_0 \ar[r]^-{-\sigma_n} & x_{n-1} \ar[d]^-{g_i} \arr^-{-\sigma_{n-1}} & x_{n-2} \ar@{..}[r] & \\
\ar@{~}[r]_-{\delta} & z \ar[r]_-{\tau_L}                   & x_{n-1} \arr_-{\sigma_{n-1}}           & x_{n-2} \ar@{..}[r] & x_2 \arr_-{\sigma_2}  & x_1 \arr_-{\sigma_1}            & x_0 \ar[r]_-{\omega_k}             & y_k \ar@{..}[r]               & x_{i,0} \ar[r]_-{\omega_i}       & y_i \arr                                 & x_{i,n-2} \ar@{..}[r] &
}}
\] 
In this case, iterated application of Step 1 yields the unfolded diagram below.
\begin{equation} \label{case-3}
\scalebox{1}{
\xymatrix@!R=4px{
\ar@{~}[r]_-{\delta} & z \ar[rrrrrr]_-{-\tau_L g_k} \ar@/_3pc/[rrrrrrrr]_-{-\tau_L g_i} &    &   &   &  &  & y_k \ar@{..}[r]  & x_{i,0} \ar[r]^-{\omega_i} & y_i \arr  & x_{i,n-2} \ar@{..}[r] &
}}
\end{equation}

If $\omega_k$ is an inverse or empty homotopy letter, then $k = 1$ and the starting unfolded diagram is given below. 
\[
\xymatrix@!R=4px{
x_1 \arr^-{-\sigma_1} & x_0 \ar[r]^-{-\sigma_n} \ar[d]_-{f_L} & x_{n-1} \arr^-{-\sigma_{n-1}} \ar@{=}[d] & x_{n-2} \ar@{..}[r] & x_2 \arr^-{-\sigma_2}  & x_1 \arr^-{-\sigma_1} \ar@{=}[d] & x_0 \ar[r]^-{-\sigma_n} \ar@{=}[d] & x_{n-1}  \ar@{..}[r] &  \\
\ar@{~}[r]_-{\delta} & z \ar[r]_-{\tau_L}                   & x_{n-1} \arr_-{\sigma_{n-1}}           & x_{n-2} \ar@{..}[r] & x_2 \arr_-{\sigma_2}  & x_1 \arr_-{\sigma_1}            & x_0 \ar@{<-}[r]_-{\omega_1}             & y_1 \ar@{..}[r]               & 
}
\]
Now, iterated application of Step 1 gives the unfolded diagram, 
\begin{equation} \label{case-4}
\xymatrix@!R=4px{
\ar@{~}[r]_-{\delta} & z  &    &   &   &  &  & \ar[llllll]^-{-\inv{f_L}\omega_1} y_1 \ar@{..}[r]  & 
}.
\end{equation}
We observe that $\inv{f_L}\omega_1 \neq 0$: for if $\sigma_1$ is direct, then $\sigma_1 f_L = 0$ because $\sigma_n = f_L \tau_L$, whence $\inv{\omega_1} f_L \neq 0$. Similarly, if $\sigma_1$ is inverse, then $\sigma_1 \omega_1 = 0$ means that $\inv{\omega_1} f_L \neq 0$ by gentleness of $(Q,I)$.

\smallskip
\noindent \textbf{Step 3.} {\it Homotopy into a string complex.}
\smallskip

We note that the unfolded diagram \eqref{case-4} is already, up to a sign, the unfolded diagram of a string complex. It is straightforward to construct an isomorphism between the string complex given by the unfolded diagram \eqref{case-4} and the string complex given by the unfolded diagram without any signs. 
For the remaining cases, one can define the required isomorphism through an unfolded diagram. Below, we explicitly give the construction for \eqref{case-1}; cases \eqref{case-2} and \eqref{case-3} are analogous.

Let $\N = (N^n, d^n_{\N})$ be the complex in $\KminusL$ corresponding to the unfolded diagram \eqref{case-1}.
For case \eqref{case-1}, the unfolded diagram of the corresponding string complex is
\begin{equation} \label{case-1a}
\scalebox{1}{
\xymatrix@!R=4px{
\ar@{~}[r]^-{\delta} & z \ar[rrrrrr]^-{\tau_L g_k}  &    &   &   &  &  & y_k \ar@{..}[r]  & x_{i,0} \ar@{<-}[r]^-{\omega_i} & y_i \arr^-{\theta_i}  & x_{i,n-2} \ar@{..}[r] & 
}.}
\end{equation}
Write $\P = (P^n, d^n_{\P})$ for the complex in $\KminusL$ whose unfolded diagram is given by \eqref{case-1a}. We now define an isomorphism $\phi^\bullet \colon \N \to \P$ whose nonzero components are defined via the unfolded diagram given in Figure~\ref{fig:case-1}. Recall that $\phi^\bullet$ is a family of morphisms consisting of matrices $\phi^n \colon N^n \to P^n$ whose entries are (possibly trivial and signed) paths in $(Q,I)$.

In Figure~\ref{fig:case-1}, the morphisms occurring as diagonal entries in the $\phi^n$ are indicated by vertical equals signs. Those entries which are supported on the indecomposable projective modules occurring in the subword $\delta$ carry the sign $-1$; those supported at all other indecomposable projective modules are identity morphisms. In this way we see that each $\phi^n$ has no zero entries on its diagonal, and each entry is either $1$ or $-1$.

To ensure commutativity of the diagram, we need to introduce some correction terms off the diagonals of $\phi^n$. These terms are indicated by red arrows marked $-1$ in the Figure~\ref{fig:case-1}. By examining Figure~\ref{fig:case-1}, one observes that the diagram commutes (meaning that $\phi^\bullet \colon \N \to \P$ is a well-defined morphism of complexes); we note that if $g_1 \neq {\rm id}_{x_{n-1}}$ then $\inv{\sigma_n} = g_1 \inv{\omega_1}$, and if $\theta_1$ is direct then $g_1 \theta_1 =0$, each by definition of the graph map giving the overlap (and gentleness of $(Q,I)$). If $g_1 = {\rm id}_{x_{n-1}}$, then the off-diagonal components continue further to the right, each one carrying a minus sign. In particular, each off-diagonal entry in $\phi^\bullet$ corresponds to the components of the graph map $\f$ with those incident with one copy of the band deleted.

Now arguing on the rank shows that for each $n$, $\phi^n$ is a full rank matrix, meaning that $\phi^\bullet \colon \N \to \P$ is an isomorphism. Alternatively, using unfolded diagrams as in Figure~\ref{fig:case-1}, one can explicitly write down the inverse of $\phi^\bullet$. This completes the proof of Theorem~\ref{thm:graph-band-to-string}\eqref{graph:longer}. \qedhere
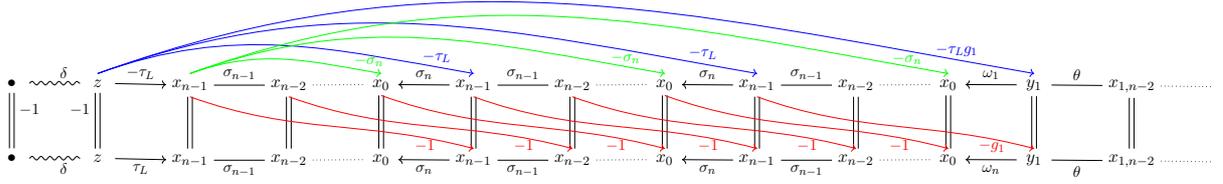
\begin{figure}[H]
\begin{tikzpicture}
\node[scale=.59] at (0,0){$
\begin{tikzpicture}
\matrix (m) [matrix of math nodes,row sep=3em,column sep=3em,minimum width=2em]
  {
     \bullet & z & x_{n-1} & x_{n-2} & x_0 & x_{n-1} & x_{n-2} & x_0 & x_{n-1} & x_{n-2} & x_0 & y_1 & x_{1,n-2} &  \quad\\
     \bullet & z & x_{n-1} & x_{n-2} & x_0 & x_{n-1} & x_{n-2} & x_0 & x_{n-1} & x_{n-2} & x_0 & y_1 & x_{1,n-2} &  \quad\\};
     
\draw[-,decorate,decoration={snake,amplitude=.4mm,segment length=2mm}] 
(m-1-1)-- node [midway,anchor=south west,scale=.9]{$\delta$} (m-1-2);
\draw[-,decorate,decoration={snake,amplitude=.4mm,segment length=2mm}] 
(m-2-1)-- node [anchor=north west,scale=.9]{$\delta$} (m-2-2);     

\draw[->]  (m-1-2)-- node [above,scale=.9]{$-\tau_L$} (m-1-3);
\draw[->]  (m-2-2)-- node [below,scale=.9]{$\tau_L$} (m-2-3);

\draw (m-1-3)-- node [above,scale=.9]{$\sigma_{n-1}$} (m-1-4);
\draw (m-2-3)-- node [below,scale=.9]{$\sigma_{n-1}$} (m-2-4);

\draw[dotted] (m-1-4)-- node [above,scale=.9]{} (m-1-5);
\draw[dotted] (m-2-4)-- node [below,scale=.9]{} (m-2-5);

\draw[<-] (m-1-5)-- node [above,scale=.9]{$\sigma_{n}$} (m-1-6);
\draw[<-] (m-2-5)-- node [below,scale=.9]{$\sigma_{n}$} (m-2-6);

\draw[-] (m-1-6)-- node [above,scale=.9]{$\sigma_{n-1}$} (m-1-7);
\draw[-] (m-2-6)-- node [below,scale=.9]{$\sigma_{n-1}$} (m-2-7);

\draw[dotted] (m-1-7)-- node [above,scale=.9]{} (m-1-8);
\draw[dotted] (m-2-7)-- node [below,scale=.9]{} (m-2-8);

\draw[<-] (m-1-8)-- node [above,scale=.9]{$\sigma_{n}$} (m-1-9);
\draw[<-] (m-2-8)-- node [below,scale=.9]{$\sigma_{n}$} (m-2-9);

\draw[-] (m-1-9)-- node [above,scale=.9]{$\sigma_{n-1}$} (m-1-10);
\draw[-] (m-2-9)-- node [below,scale=.9]{$\sigma_{n-1}$} (m-2-10);

\draw[dotted] (m-1-10)-- node [above,scale=.9]{} (m-1-11);
\draw[dotted] (m-2-10)-- node [below,scale=.9]{} (m-2-11);

\draw[<-] (m-1-11)-- node [above,scale=.9]{$\omega_{1}$} (m-1-12);
\draw[<-] (m-2-11)-- node [below,scale=.9]{$\omega_{n}$} (m-2-12);

\draw[-] (m-1-12)-- node [above,scale=.9]{$\theta$} (m-1-13);
\draw[-] (m-2-12)-- node [below,scale=.9]{$\theta$} (m-2-13);

\draw[dotted] (m-1-13)-- node [above,scale=.9]{} (m-1-14);
\draw[dotted] (m-2-13)-- node [below,scale=.9]{} (m-2-14);

\draw[transform canvas={xshift=-1.5pt}] (m-1-1) -- (m-2-1);
\draw[transform canvas={xshift=1.5pt}] (m-1-1) -- node [anchor=south west,scale=.9]{$-1$}(m-2-1);

\draw[transform canvas={xshift=-1.5pt}] (m-1-2) -- node [anchor=south east,scale=.9]{$-1$} (m-2-2);
\draw[transform canvas={xshift=1.5pt}] (m-1-2) --  (m-2-2);

\draw[transform canvas={xshift=-1.5pt}] (m-1-3) -- (m-2-3);
\draw[transform canvas={xshift=1.5pt}] (m-1-3) --  (m-2-3);

\draw[transform canvas={xshift=-1.5pt}] (m-1-4) -- (m-2-4);
\draw[transform canvas={xshift=1.5pt}] (m-1-4) --  (m-2-4);

\draw[transform canvas={xshift=-1.5pt}] (m-1-5) -- (m-2-5);
\draw[transform canvas={xshift=1.5pt}] (m-1-5) --  (m-2-5);

\draw[transform canvas={xshift=-1.5pt}] (m-1-6) -- (m-2-6);
\draw[transform canvas={xshift=1.5pt}] (m-1-6) --  (m-2-6);

\draw[transform canvas={xshift=-1.5pt}] (m-1-7) -- (m-2-7);
\draw[transform canvas={xshift=1.5pt}] (m-1-7) --  (m-2-7);

\draw[transform canvas={xshift=-1.5pt}] (m-1-8) -- (m-2-8);
\draw[transform canvas={xshift=1.5pt}] (m-1-8) --  (m-2-8);

\draw[transform canvas={xshift=-1.5pt}] (m-1-9) -- (m-2-9);
\draw[transform canvas={xshift=1.5pt}] (m-1-9) --  (m-2-9);

\draw[transform canvas={xshift=-1.5pt}] (m-1-10) -- (m-2-10);
\draw[transform canvas={xshift=1.5pt}] (m-1-10) --  (m-2-10);

\draw[transform canvas={xshift=-1.5pt}] (m-1-11) -- (m-2-11);
\draw[transform canvas={xshift=1.5pt}] (m-1-11) --  (m-2-11);

\draw[transform canvas={xshift=-1.5pt}] (m-1-12) -- (m-2-12);
\draw[transform canvas={xshift=1.5pt}] (m-1-12) --  (m-2-12);

\draw[transform canvas={xshift=-1.5pt}] (m-1-13) -- (m-2-13);
\draw[transform canvas={xshift=1.5pt}] (m-1-13) --  (m-2-13);

 
\draw [->,blue] (m-1-2.north) to [out=20,in=170]  node[above left=-.5cm and -85.2pt,anchor=south west,scale=.9]{$-\tau_L$}  (m-1-6.north);

\draw [->,blue] (m-1-2.north) to [out=20,in=170]  node[above left=-.9cm and -165.2pt,anchor=south west,scale=.9]{$-\tau_L$}  (m-1-9.north);

\draw [->,blue] (m-1-2.north) to [out=20,in=170]  node[above left=-1.3cm and -235.2pt,anchor=south west,scale=.9]{$-\tau_Lg_1$}  (m-1-12.north);

\draw [->,green] (m-1-3.north) to [out=20,in=170]  node[above left=-.28cm and -41pt,anchor=south west,scale=.9]{$-\sigma_n$}  (m-1-5.north);

\draw [->,green] (m-1-3.north) to [out=20,in=170]  node[above left=-.69cm and -115pt,anchor=south west,scale=.9]{$-\sigma_n$}  (m-1-8.north);

\draw [->,green] (m-1-3.north) to [out=20,in=170]  node[above left=-1.19cm and -205pt,anchor=south west,scale=.9]{$-\sigma_n$}  (m-1-11.north);

\draw [->,red] (m-1-3.south) to [out=-20,in=165]  node[above left=-.2cm and -50pt,anchor=north west,scale=.9]{$-1$}  (m-2-6.north);

\draw [->,red] (m-1-4.south) to [out=-20,in=165]  node[above left=-.2cm and -50pt,anchor=north west,scale=.9]{$-1$}  (m-2-7.north);

\draw [->,red] (m-1-5.south) to [out=-20,in=165]  node[above left=-.2cm and -50pt,anchor=north west,scale=.9]{$-1$}  (m-2-8.north);

\draw [->,red] (m-1-6.south) to [out=-20,in=165]  node[above left=-.2cm and -50pt,anchor=north west,scale=.9]{$-1$}  (m-2-9.north);

\draw [->,red] (m-1-7.south) to [out=-20,in=165]  node[above left=-.2cm and -50pt,anchor=north west,scale=.9]{$-1$}  (m-2-10.north);

\draw [->,red] (m-1-8.south) to [out=-20,in=165]  node[above left=-.2cm and -50pt,anchor=north west,scale=.9]{$-1$}  (m-2-11.north);

\draw [->,red] (m-1-9.south) to [out=-20,in=165]  node[above left=-.2cm and -50pt,anchor=north west,scale=.9]{$-g_1$}  (m-2-12.north);
\end{tikzpicture}
$};
\end{tikzpicture}
\caption{Unfolded diagram defining the isomorphism $\phi^\bullet \colon \N \to \P$. The blue and green arrows represent components of the differential $d^\bullet_N$ that we wish to remove via the isomorphism. Morphisms in grey are diagonal entries; morphisms in red are off diagonal entries, if $g_1 = {\rm id}$ then they may continue to the right.} \label{fig:case-1}
\end{figure}
\end{proof}

Below we state a dual version of Theorem~\ref{thm:graph-band-to-string} in which $\sigma$ is homotopy string and $\tau$ is a homotopy band that extends \cite[Prop. 2.12]{CPS1}; the proof is analogous.

\begin{theorem} \label{thm:graph-string-to-band}
Let $\Lambda$ be a gentle algebra, $\sigma$ be a homotopy string and $\tau$ be a homotopy band. For $\mu \in \kk^*$,  let $\f \colon \P_\sigma \to \B_{\tau,\mu}$ be a graph map determined by a (possibly trivial) maximal common homotopy string $\rho$.
\begin{enumerate}
\item {\rm (\cite[Prop. 2.12]{CPS1})} Suppose $\rho$ is a proper subword of $\tau$. Then, after suitable rotation of $\tau$, there is a decomposition $\sigma = \beta \sigma_L \rho \sigma_R \alpha$ and a decomposition $\tau = \rho \gamma$. Then $\M_{\f} \cong \P_c$, where $c = \beta \sigma_L \inv{\gamma} \sigma_R \alpha$. This is indicated in the unfolded diagram below, in which $\gamma = \gamma_\ell \cdots \gamma_1$. Note that, if $f_L \neq \emptyset$ then $\sigma_L \inv{\gamma_1} = f_L$; similarly, if $f_R \neq \emptyset$ then $\inv{\gamma_\ell} \sigma_R = \inv{f_R}$.
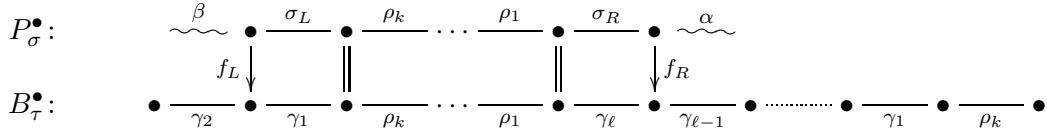
\[
\xymatrix@!R=4px{
\P_\sigma \colon & \ar@{~}[r]^-{\beta}   & \xydot \arr^-{\sigma_L} \ar[d]_-{f_L}  & \xydot \arr^-{\rho_k} \ar@{=}[d]  &  \cdots \arr^-{\rho_1}  & \xydot \arr^-{\sigma_R}  \ar@{=}[d]   & \xydot \ar@{~}[r]^-{\alpha} \ar[d]^-{f_R} &  & & & \\
\B_\tau \colon     & \xydot \arr_-{\gamma_2} & \xydot \arr_-{\gamma_1}                          & \xydot \arr_-{\rho_k}                   & \cdots \arr_-{\rho_1}   & \xydot \arr_-{\gamma_\ell}                           & \xydot \arr_-{\gamma_{\ell - 1}}                      & \xydot \ar@{..}[r] & \xydot \arr_-{\gamma_1} & \xydot \arr_-{\rho_k} & \xydot 
}
\]
\item Suppose there is an integer $\ell \geq 1$ such that $\rho$ is a proper subword of $\tau^{\ell+1}$ and $\tau^\ell$ is a (not necessarily proper) subword of $\rho$. Then, after suitable rotation of $\tau$, there is a decomposition $\tau = \delta \gamma$, with $\delta$ or $\gamma$ possibly trivial, such that 
\[
\sigma = \beta \tau^\ell \delta \alpha \text{ for some } \ell \geq 1 \text{ and some homotopy strings } \alpha \text{ and } \beta. 
\]
In this case, $\M_{\f} = \P_c$, where $c = \beta \tau^{\ell - 1} \delta \alpha$.
\end{enumerate}
\end{theorem}

Finally, there is a version of Theorem~\ref{thm:graph-band-to-string} in which both $\sigma$ and $\tau$ are homotopy bands. Before stating it, we need a lemma explaining how a complex corresponding to an unfolded diagram given by a power of a band decomposes into indecomposable complexes. 

\begin{lemma} \label{lem:power-of-band}
Let $\kk$ be an algebraically closed field. Suppose $\tau$ is a homotopy band and $\sigma = \tau^n$. Let $\B_{\sigma,\lambda}$ be the complex induced by the unfolded diagram of $\sigma$ in which the scalar $\lambda \in \kk$ is placed on direct homotopy letter. Then $\B_{\sigma,\lambda} \cong \bigoplus_{i=1}^n \B_{\tau, \omega^i \mu}$, where $\mu = \sqrt[n]{\lambda}$ is some fixed $n^{\it th}$ root of $\lambda$ and $\omega$ is a primitive $n^{\it th}$ root of unity.
\end{lemma}

\begin{proof}
The unfolded diagram in Figure~\ref{fig:split-mono} shows how to define a split monomorphism $\B_{\tau, \omega \mu} \to \B_{\sigma,\lambda}$; the split monomorphism $\B_{\tau, \omega^i \mu} \to \B_{\sigma,\lambda}$ is defined analogously. The direct sum of these split monomorphisms is clearly an isomorphism (its inverse is given by the direct sum of the corresponding split epimorphisms). \qedhere
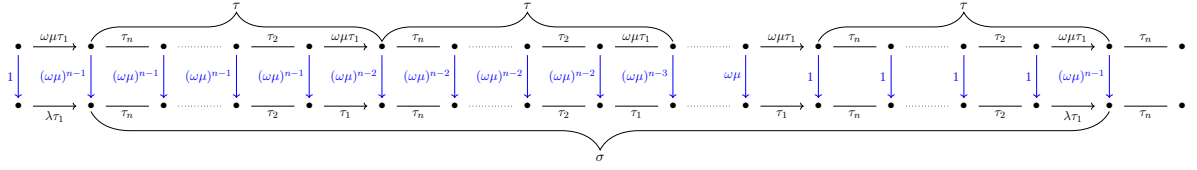
\begin{figure}[H]
\begin{tikzpicture}
\node[scale=.5] at (0,0){$
\begin{tikzpicture}
\matrix (m) [matrix of math nodes,row sep=3em,column sep=3em,minimum width=2em]
  {
     \bullet & \bullet & \bullet & \bullet & \bullet & \bullet & \bullet & \bullet & \bullet & \bullet & \bullet & \bullet & \bullet & \bullet & \bullet & \bullet & \bullet  \\
     \bullet & \bullet & \bullet & \bullet & \bullet & \bullet & \bullet & \bullet & \bullet & \bullet & \bullet & \bullet & \bullet & \bullet & \bullet & \bullet & \bullet  \\};
\draw[->]  (m-1-1)-- node [above,scale=.9]{$\omega\mu\tau_1$} (m-1-2);
\draw[-]  (m-1-2)-- node [above,scale=.9]{$\tau_n$} (m-1-3);
\draw[dotted]  (m-1-3)-- node [above,scale=.9]{} (m-1-4);
\draw[-]  (m-1-4)-- node [above,scale=.9]{$\tau_2$} (m-1-5);
\draw[->]  (m-1-5)-- node [above,scale=.9]{$\omega\mu\tau_1$} (m-1-6);
\draw[-]  (m-1-6)-- node [above,scale=.9]{$\tau_n$} (m-1-7);
\draw[dotted]  (m-1-7)-- node [above,scale=.9]{} (m-1-8);
\draw[-]  (m-1-8)-- node [above,scale=.9]{$\tau_2$} (m-1-9);
\draw[-]  (m-1-9)-- node [above,scale=.9]{$\omega\mu\tau_1$} (m-1-10);
\draw[dotted]  (m-1-10)-- node [above,scale=.9]{} (m-1-11);
\draw[->]  (m-1-11)-- node [above,scale=.9]{$\omega\mu\tau_1$} (m-1-12);
\draw[-]  (m-1-12)-- node [above,scale=.9]{$\tau_n$} (m-1-13);
\draw[dotted]  (m-1-13)-- node [above,scale=.9]{} (m-1-14);
\draw[-]  (m-1-14)-- node [above,scale=.9]{$\tau_2$} (m-1-15);
\draw[->]  (m-1-15)-- node [above,scale=.9]{$\omega\mu\tau_1$} (m-1-16);
\draw[-]  (m-1-16)-- node [above,scale=.9]{$\tau_n$} (m-1-17);
\draw[->]  (m-2-1)-- node [below,scale=.9]{$\lambda\tau_1$} (m-2-2);
\draw[-]  (m-2-2)-- node [below,scale=.9]{$\tau_n$} (m-2-3);
\draw[dotted]  (m-2-3)-- node [below,scale=.9]{} (m-2-4);
\draw[-]  (m-2-4)-- node [below,scale=.9]{$\tau_2$} (m-2-5);
\draw[->]  (m-2-5)-- node [below,scale=.9]{$\tau_1$} (m-2-6);
\draw[-]  (m-2-6)-- node [below,scale=.9]{$\tau_n$} (m-2-7);
\draw[dotted] (m-2-7)-- node [below,scale=.9]{} (m-2-8);
\draw[-]  (m-2-8)-- node [below,scale=.9]{$\tau_2$} (m-2-9);
\draw[-]  (m-2-9)-- node [below,scale=.9]{$\tau_1$} (m-2-10);
\draw[dotted] (m-2-10)-- node [below,scale=.9]{} (m-2-11);
\draw[->]  (m-2-11)-- node [below,scale=.9]{$\tau_1$} (m-2-12);
\draw[-]  (m-2-12)-- node [below,scale=.9]{$\tau_n$} (m-2-13);
\draw[dotted] (m-2-13)-- node [below,scale=.9]{} (m-2-14);
\draw[-]  (m-2-14)-- node [below,scale=.9]{$\tau_2$} (m-2-15);
\draw[->]  (m-2-15)-- node [below,scale=.9]{$\lambda\tau_1$} (m-2-16);
\draw[-]  (m-2-16)-- node [below,scale=.9]{$\tau_n$} (m-2-17);
\draw [decorate,decoration={brace,amplitude=22pt,raise=4pt},xshift=0pt,yshift=1.2pt]
(m-1-2.center) -- (m-1-6.center)node [above,midway,yshift=25pt] {$\tau$};
\draw [decorate,decoration={brace,amplitude=22pt,raise=4pt},xshift=0pt,yshift=1.2pt]
(m-1-6.center) -- (m-1-10.center)node [above,midway,yshift=25pt] {$\tau$};
\draw [decorate,decoration={brace,amplitude=22pt,raise=4pt},xshift=0pt,yshift=1.2pt]
(m-1-12.center) -- (m-1-16.center)node [above,midway,yshift=25pt] {$\tau$};
\draw [decorate,decoration={brace,amplitude=29pt,mirror,raise=4pt},xshift=0pt,yshift=1.2pt]
(m-2-2.center) -- (m-2-16.center)node [above,midway,yshift=-45pt] {$\sigma$};

\draw [->,blue] (m-1-1.south) to node[above,anchor=east,scale=.9]{$1$}  (m-2-1.north);
\draw [->,blue] (m-1-2.south) to node[below,anchor=east,scale=.9]{$(\omega\mu)^{n-1}$}  (m-2-2.north);
\draw [->,blue] (m-1-3.south) to node[below,anchor=east,scale=.9]{$(\omega\mu)^{n-1}$}  (m-2-3.north);
\draw [->,blue] (m-1-4.south) to node[below,anchor=east,scale=.9]{$(\omega\mu)^{n-1}$}  (m-2-4.north);
\draw [->,blue] (m-1-5.south) to node[below,anchor=east,scale=.9]{$(\omega\mu)^{n-1}$}  (m-2-5.north);
\draw [->,blue] (m-1-6.south) to node[below,anchor=east,scale=.9]{$(\omega\mu)^{n-2}$}  (m-2-6.north);
\draw [->,blue] (m-1-7.south) to node[below,anchor=east,scale=.9]{$(\omega\mu)^{n-2}$}  (m-2-7.north);
\draw [->,blue] (m-1-8.south) to node[below,anchor=east,scale=.9]{$(\omega\mu)^{n-2}$}  (m-2-8.north);
\draw [->,blue] (m-1-9.south) to node[below,anchor=east,scale=.9]{$(\omega\mu)^{n-2}$}  (m-2-9.north);
\draw [->,blue] (m-1-10.south) to node[below,anchor=east,scale=.9]{$(\omega\mu)^{n-3}$}  (m-2-10.north);
\draw [->,blue] (m-1-11.south) to node[below,anchor=east,scale=.9]{$\omega\mu$}  (m-2-11.north);
\draw [->,blue] (m-1-12.south) to node[below,anchor=east,scale=.9]{$1$}  (m-2-12.north);
\draw [->,blue] (m-1-13.south) to node[below,anchor=east,scale=.9]{$1$}  (m-2-13.north);
\draw [->,blue] (m-1-14.south) to node[below,anchor=east,scale=.9]{$1$}  (m-2-14.north);
\draw [->,blue] (m-1-15.south) to node[below,anchor=east,scale=.9]{$1$}  (m-2-15.north);
\draw [->,blue] (m-1-16.south) to node[below,anchor=east,scale=.9]{$(\omega\mu)^{n-1}$}  (m-2-16.north);
\end{tikzpicture}
$};
\end{tikzpicture}
\caption{Unfolded diagram of a split monomorphism $\B_{\tau, \omega \mu} \to \B_{\sigma, \lambda}$, where the copies of $\tau$ on the top row are identified, and copies of $\sigma$ on the bottom row are identified.} \label{fig:split-mono}
\end{figure}
\end{proof}

The first three statements of the following theorem extend \cite[Prop. 2.9]{CPS1} to the case of a graph map between band complexes in which the overlap is longer than at least one of the homotopy bands. Again, the effect on the mapping cone calculus is `to remove the shorter homotopy band from the longer homotopy band'. The final statement uses Lemma~\ref{lem:power-of-band} to correct the statement of \cite[Prop. 2.9]{CPS1} in the case that the overlap is shorter than both homotopy bands.

\begin{theorem} \label{thm:graph-band-to-band}
Let $\Lambda$ be a gentle algebra over an algebraically closed field $\kk$. Let $(\sigma, \lambda)$ and $(\tau, \mu)$ be homotopy bands, where by convention $\lambda$ and $\mu$ are placed on direct homotopy letters. Suppose $\f \colon \B_{\sigma,\lambda} \to \B_{\tau,\mu}$ is a graph map determined by a (possibly trivial) maximal common homotopy string $\rho$.
\begin{enumerate}
\item \label{case:1} Suppose $\length{\tau} = \length{\sigma}$ and both $\sigma$ and $\tau$ are subwords of $\rho$. Then $\f$ is an isomorphism and $\M_{\f} \cong 0^\bullet$.
\item \label{case:2} Suppose $\length{\tau} < \length{\sigma}$ and $\tau$ is a subword of $\rho$. Then there is a homotopy band $\theta$ and an integer $k \geq 1$ such that $\sigma = \tau \theta^k$ and 
\[
\M_{\f} \cong \bigoplus_{i=1}^k \B_{\theta,\omega^i \sqrt[k]{-\lambda \mu}},
\]
where $\omega$ is a primitive $k^{\it th}$ root of unity.
\item \label{case:3} Suppose $\length{\tau} > \length{\sigma}$ and $\sigma$ is a subword of $\rho$. Then there is a homotopy band $\phi$ and an integer $\ell \geq 1$ such that $\tau = \sigma \phi^\ell$ and 
\[
\M_{\f} \cong \bigoplus_{i=1}^\ell \B_{\phi, \epsilon^i \sqrt[\ell]{-\lambda \mu}},
\]
where  $\epsilon$ is a primitive $\ell^{\it th}$ root of unity.
\item \label{case:4} Suppose that $\rho$ is a proper subword of both $\sigma$ and $\tau$. After suitable rotations of $\sigma$ and $\tau$, we have $\sigma = \rho \alpha$ and $\tau = \rho \gamma$. Then, there exists a homotopy band $\theta$ and an integer $k \geq 1$ such that $\inv{\gamma}\alpha = \theta^k$ and
\[
\M_{\f} = 
\begin{cases}
\bigoplus_{i = 1}^k \B_{\theta, \omega^i \sqrt[k]{\lambda \mu}} & \text{ if } \length{\rho} \text{ is even; and,} \\
\bigoplus_{i = 1}^k \B_{\theta, \omega^i \sqrt[k]{-\lambda \mu}} &  \text{ if } \length{\rho} \text{ is odd,}
\end{cases}
\]
where $\omega$ is a primitive $k^{\it th}$ root of unity.
\end{enumerate}
\end{theorem}

\begin{proof}
For \eqref{case:1}, $\length{\tau} = \length{\sigma}$ and both $\sigma$ and $\tau$ being subwords of $\rho$ means that in order for the corresponding unfolded diagram to be commutative, $\rho$ must be an infinite overlap. By Corollary~\ref{cor:infinite}, it follows that $\f$ is an isomorphism. In particular, for the remainder of the proof we may assume without loss of generality that $\rho$ is finite.
 
For \eqref{case:2}, the assumption $\length{\tau} < \length{\sigma}$ implies that, after suitable rotation, $\rho = \tau \rho'$ and $\sigma = \tau \alpha$, with $\alpha$ nontrivial and $\rho'$ possibly trivial. We have $s(\alpha) = s(\sigma) = e(\sigma) = e(\tau) = s(\tau) = e(\alpha)$, showing that $\alpha$ starts and ends at the same vertex in $Q$. The same argument applied to degrees shows that $\alpha$ starts and ends in the same degree. 

Write $\sigma = \sigma_m \cdots \sigma_1$, $\tau = \tau_n \cdots \tau_1$ and $\alpha = \alpha_t \cdots \alpha_1$. To conclude that $\alpha$ is a power of a band, we must show that $\alpha_1 \alpha_t$ is defined as a homotopy band, i.e. they remain distinct homotopy letters. 
If $\tau$ is a proper subword of $\rho$, this is clear: then $\rho' = \alpha_t \rho''$ for some possibly trivial homotopy string $\rho''$. But since $\alpha_t \rho'$ is a subword of ${}^\infty \tau^\infty$, we have that $\alpha_t = \tau_n = \sigma_m$. But since $\alpha_1 = \sigma_1$ and $\sigma$ is a homotopy band, it follows that $\alpha_1 \alpha_t$ is defined as a homotopy band.

Now suppose that $\tau$ is not a proper subword of $\rho$, i.e. $\tau = \rho$. In this case, we have the following unfolded diagram of the graph map $\f$ in which $f_L$ and $f_R$ are possibly empty.
\[
\xymatrix@!R=4px{
\ar@{.}[r] & \xydot \ar[d]_-{f_L} \arr^-{\alpha_1}  & x \ar@{=}[d] \arr^-{\sigma_m} & \xydot \ar@{=}[d] \ar@{.}[r] & \xydot \ar@{=}[d] \arr^-{\sigma_{t+1}} & x \ar@{=}[d] \arr^-{\alpha_t} & \xydot \ar[d]^-{f_R} \ar@{.}[r] & \\
\ar@{.}[r] & \xydot \arr_-{\tau_1}                         & x \arr_-{\tau_n}                        & \xydot \ar@{.}[r]                  & \xydot \arr_-{\tau_1}                             & x \arr_-{\tau_n}                     & \xydot \ar@{.}[r] &
}
\]
If $\alpha_1$ and $\alpha_t$ are both direct, then $\alpha_1 = f_L \tau_1$ and $\tau_n = \alpha_t f_R$. Thus $\tau_1 \tau_n$ being defined implies that $\alpha_1 \alpha_t$ is defined. Dually when $\alpha_1$ and $\alpha_t$ are both inverse.

Now suppose $\alpha_1$ is inverse and $\alpha_t$ is direct; the argument when $\alpha_1$ is direct and $\alpha_t$ is inverse is dual. There are four cases, which we treat below.
\begin{enumerate}[label=(\roman*)]
\item Suppose $\tau_1$ and $\tau_n$ are both direct. Then $\sigma_m = \tau_n = \alpha_t f_R$, giving us the unfolded diagram
\[
\xymatrix{ \ar@{.}[r] & \xydot \ar@{<-}[r]^-{\alpha_1 = \sigma_1} & \xydot \ar[r]^-{\sigma_m = \alpha_t f_R} & \xydot \ar@{.}[r] & }
\]
since $\sigma_1 \sigma_m$ is defined, it follows that so is $\alpha_1 \alpha_t$.
\item The case that $\tau_1$ and $\tau_n$ are both inverse is dual to the one above.
\item Suppose $\tau_1$ is inverse and $\tau_n$ is direct. Then $\tau_n = \alpha_t f_R$ and $\tau_1 = \inv{f_L} \alpha_1$, giving us the unfolded diagram
\[
\xymatrix{ \ar@{.}[r] & \xydot \ar@{<-}[r]^-{\tau_1 = \inv{f_L} \alpha_1} & \xydot \ar[r]^-{\tau_n = \alpha_t f_R} & \xydot \ar@{.}[r] & }
\]
since $\tau_1 \tau_n$ is defined, it follows that so is $\alpha_1 \alpha_t$.
\item Suppose $\tau_1$ is direct and $\tau_n$ is inverse. Now the fact that $\tau_1 \alpha_t = 0$, $\alpha_1 \tau_n = 0$ and $\tau_1 \tau_n$ is defined together with gentleness means that $\alpha_1 \alpha_t$ must also be defined.
\end{enumerate}

Hence, we find that $\alpha$ is (a power of a) homotopy band, i.e. $\alpha = \theta^k$ for some homotopy band $\theta$ and some integer $k \geq 1$.

Now, arguing as in Theorem~\ref{thm:graph-band-to-string} using Lemma~\ref{lem:homotopy} shows that $\M_{\f} \cong \B_{\alpha,-\lambda \mu}$, where $\B_{\alpha, -\lambda \mu}$ denotes the complex induced by the unfolded diagram for $\alpha$ with $-\lambda \mu$ placed on a direct homotopy letter. Applying Lemma~\ref{lem:power-of-band} shows that 
$\M_{\f} \cong \bigoplus_{i=1}^k \B_{\theta,\omega^i \sqrt[k]{-\lambda \mu}}$, where $\omega$ is a primitive $k^{\rm th}$ of unity.

Statement \eqref{case:3} is dual to statement \eqref{case:2}. 

Finally, for \eqref{case:4}, the argument given in the proof of \cite[Prop. 2.9]{CPS1} shows that $\M_{\f} \cong \B_{\inv{\gamma}\alpha,\pm \lambda\mu}$, where $\B_{\inv{\gamma}\alpha,\pm \lambda\mu}$ is the complex induced by the unfolded diagram for $\inv{\gamma} \alpha$ with $\pm \lambda \mu$ placed on a direct homotopy letter. The argument in the proof of \cite[Prop. 2.9]{CPS1} shows that $\inv{\gamma}\alpha$ is a (not necessarily trivial) power of a homotopy band, i.e. $\inv{\gamma}\alpha = \theta^k$ for some homotopy band $\theta$ and some integer $k \geq 1$. Thus, it now follows from Lemma~\ref{lem:power-of-band} that $\M_{\f} \cong \bigoplus_{i=1}^k \B_{\theta,\omega^i \sqrt[k]{\pm\lambda \mu}}$, where $\omega$ is again a primitive $k^{\rm th}$ of unity.
\end{proof}

\begin{example} \label{ex:not-a-power}
Let $\sigma$ and $\tau$ be the homotopy bands given in Example~\ref{ex:two-kroneckers} and $f\colon \B_{\tau,1}\to \B_{\sigma,1}$ be the graph map defined by the unfolded diagram \eqref{unfolded}. Note that, in this example, the roles of $\sigma$ and $\tau$ are interchanged in comparison with Theorem~\ref{thm:graph-band-to-band}. We have $\length{\tau} < \length{\sigma}$, putting us in case \eqref{case:3} of the theorem. In this case we have $\sigma = \tau d \inv{c} \, \inv{a} b$. 
In particular, $\ell = 1$ and $\phi = d \inv{c} \, \inv{a} b$ is a homotopy band. Hence, $\M_{\f} \cong \B_{\phi,-1}$ and the mapping cone is indecomposable.
\end{example}

\begin{example} \label{ex:cubed}
Let $\kk = \bC$ and let $\Lambda$ be the $\bC$-algebra given by the quiver with relations in Example~\ref{ex:two-kroneckers}. Let $\sigma$ and $\tau$ be the following homotopy bands
$\sigma = (d \inv{c})^7 \inv{a} b$ and $\tau = (d \inv{c})^4 \inv{a} b$.
Suppose $\f \colon \B_{\sigma,-1} \to \B_{\tau, 1}$ is the graph map whose unfolded diagram is given below.
\[
\begin{tikzpicture}
\node[scale=.5] at (0,0){$
\begin{tikzpicture}
\matrix (m) [matrix of math nodes,row sep=2.4em,column sep=1.8em,minimum width=2em]
  {
     3 & 1 & 3 & 1 & 3 & 1 & 3 & 1 & 3 & 1 & 2 & 1 & 3 & 1 & 3 & 1 & 3 & 1 & 3 & 1 & 3 & 1 \\
     2 & 1 & 3 & 1 & 3 & 1 & 3 & 1 & 3 & 1 & 2 & 1 & 3 & 1 & 3 & 1 & 3 & 1 & 3 & 1 & 2 & 1  
     \\};
\draw[<-]  (m-1-1)-- node [above,scale=.9]{$\inv{c}$} (m-1-2);
\draw[->]  (m-1-2)-- node [above,scale=.9]{$d$} (m-1-3);
\draw[<-]  (m-1-3)-- node [above,scale=.9]{$\inv{c}$} (m-1-4);
\draw[->]  (m-1-4)-- node [above,scale=.9]{$d$} (m-1-5);
\draw[<-]  (m-1-5)-- node [above,scale=.9]{$\inv{c}$} (m-1-6);
\draw[->]  (m-1-6)-- node [above,scale=.9]{$d$} (m-1-7);
\draw[<-]  (m-1-7)-- node [above,scale=.9]{$\inv{c}$} (m-1-8);
\draw[->]  (m-1-8)-- node [above,scale=.9]{$d$} (m-1-9);
\draw[<-]  (m-1-9)-- node [above,scale=.9]{$\inv{c}$} (m-1-10);
\draw[<-]  (m-1-10)-- node [above,scale=.9]{$\inv{a}$} (m-1-11);
\draw[->]  (m-1-11)-- node [above,scale=.9]{$b$} (m-1-12);
\draw[->]  (m-1-12)-- node [above,scale=.9]{$d$} (m-1-13);
\draw[<-]  (m-1-13)-- node [above,scale=.9]{$\inv{c}$} (m-1-14);
\draw[->]  (m-1-14)-- node [above,scale=.9]{$d$} (m-1-15);
\draw[<-]  (m-1-15)-- node [above,scale=.9]{$\inv{c}$} (m-1-16);
\draw[->]  (m-1-16)-- node [above,scale=.9]{$d$} (m-1-17);
\draw[<-]  (m-1-17)-- node [above,scale=.9]{$\inv{c}$} (m-1-18);
\draw[->]  (m-1-18)-- node [above,scale=.9]{$d$} (m-1-19);
\draw[<-]  (m-1-19)-- node [above,scale=.9]{$\inv{c}$} (m-1-20);
\draw[->]  (m-1-20)-- node [above,scale=.9]{$d$} (m-1-21);
\draw[<-]  (m-1-21)-- node [above,scale=.9]{$\inv{c}$} (m-1-22);
\draw[->]  (m-2-1)-- node [below,scale=.9]{$b$} (m-2-2);
\draw[->]  (m-2-2)-- node [below,scale=.9]{$d$} (m-2-3);
\draw[<-]  (m-2-3)-- node [below,scale=.9]{$\inv{c}$} (m-2-4);
\draw[->]  (m-2-4)-- node [below,scale=.9]{$d$} (m-2-5);
\draw[<-]  (m-2-5)-- node [below,scale=.9]{$\inv{c}$} (m-2-6);
\draw[->]  (m-2-6)-- node [below,scale=.9]{$d$} (m-2-7);
\draw[<-]  (m-2-7)-- node [below,scale=.9]{$\inv{c}$} (m-2-8);
\draw[->]  (m-2-8)-- node [below,scale=.9]{$d$} (m-2-9);
\draw[<-]  (m-2-9)-- node [below,scale=.9]{$\inv{c}$} (m-2-10);
\draw[<-]  (m-2-10)-- node [below,scale=.9]{$\inv{a}$} (m-2-11);
\draw[->]  (m-2-11)-- node [below,scale=.9]{$b$} (m-2-12);
\draw[->]  (m-2-12)-- node [below,scale=.9]{$d$} (m-2-13);
\draw[<-]  (m-2-13)-- node [below,scale=.9]{$\inv{c}$} (m-2-14);
\draw[->]  (m-2-14)-- node [below,scale=.9]{$d$} (m-2-15);
\draw[<-]  (m-2-15)-- node [below,scale=.9]{$\inv{c}$} (m-2-16);
\draw[->]  (m-2-16)-- node [below,scale=.9]{$d$} (m-2-17);
\draw[<-]  (m-2-17)-- node [below,scale=.9]{$\inv{c}$} (m-2-18);
\draw[->]  (m-2-18)-- node [below,scale=.9]{$d$} (m-2-19);
\draw[<-]  (m-2-19)-- node [below,scale=.9]{$\inv{c}$} (m-2-20);
\draw[<-]  (m-2-20)-- node [below,scale=.9]{$\inv{a}$} (m-2-21);
\draw[<-]  (m-2-21)-- node [below,scale=.9]{$b$} (m-2-22);
\draw [decorate,decoration={brace,amplitude=22pt,raise=4pt},xshift=0pt,yshift=1.2pt]
(m-1-2.center) -- (m-1-18.center)node [above,midway,yshift=25pt] {$\sigma$};
\draw [decorate,decoration={brace,amplitude=29pt,mirror,raise=4pt},xshift=0pt,yshift=1.2pt]
(m-2-2.center) -- (m-2-12.center)node [above,midway,yshift=-45pt] {$\tau$};

\draw[transform canvas={xshift=-1.5pt}] (m-1-2) -- node [anchor=south east,scale=.9]{} (m-2-2);
\draw[transform canvas={xshift=1.5pt}] (m-1-2) --  (m-2-2);

\draw[transform canvas={xshift=-1.5pt}] (m-1-3) -- (m-2-3);
\draw[transform canvas={xshift=1.5pt}] (m-1-3) --  (m-2-3);

\draw[transform canvas={xshift=-1.5pt}] (m-1-4) -- (m-2-4);
\draw[transform canvas={xshift=1.5pt}] (m-1-4) --  (m-2-4);

\draw[transform canvas={xshift=-1.5pt}] (m-1-5) -- (m-2-5);
\draw[transform canvas={xshift=1.5pt}] (m-1-5) --  (m-2-5);

\draw[transform canvas={xshift=-1.5pt}] (m-1-6) -- (m-2-6);
\draw[transform canvas={xshift=1.5pt}] (m-1-6) --  (m-2-6);

\draw[transform canvas={xshift=-1.5pt}] (m-1-7) -- (m-2-7);
\draw[transform canvas={xshift=1.5pt}] (m-1-7) --  (m-2-7);

\draw[transform canvas={xshift=-1.5pt}] (m-1-8) -- (m-2-8);
\draw[transform canvas={xshift=1.5pt}] (m-1-8) --  (m-2-8);

\draw[transform canvas={xshift=-1.5pt}] (m-1-9) -- (m-2-9);
\draw[transform canvas={xshift=1.5pt}] (m-1-9) --  (m-2-9);

\draw[transform canvas={xshift=-1.5pt}] (m-1-10) -- (m-2-10);
\draw[transform canvas={xshift=1.5pt}] (m-1-10) --  (m-2-10);

\draw[transform canvas={xshift=-1.5pt}] (m-1-11) -- (m-2-11);
\draw[transform canvas={xshift=1.5pt}] (m-1-11) --  (m-2-11);

\draw[transform canvas={xshift=-1.5pt}] (m-1-12) -- (m-2-12);
\draw[transform canvas={xshift=1.5pt}] (m-1-12) --  (m-2-12);

\draw[transform canvas={xshift=-1.5pt}] (m-1-13) -- (m-2-13);
\draw[transform canvas={xshift=1.5pt}] (m-1-13) --  (m-2-13);

\draw[transform canvas={xshift=-1.5pt}] (m-1-14) -- (m-2-14);
\draw[transform canvas={xshift=1.5pt}] (m-1-14) --  (m-2-14);

\draw[transform canvas={xshift=-1.5pt}] (m-1-15) -- (m-2-15);
\draw[transform canvas={xshift=1.5pt}] (m-1-15) --  (m-2-15);

\draw[transform canvas={xshift=-1.5pt}] (m-1-16) -- (m-2-16);
\draw[transform canvas={xshift=1.5pt}] (m-1-16) --  (m-2-16);

\draw[transform canvas={xshift=-1.5pt}] (m-1-17) -- (m-2-17);
\draw[transform canvas={xshift=1.5pt}] (m-1-17) --  (m-2-17);

\draw[transform canvas={xshift=-1.5pt}] (m-1-18) -- (m-2-18);
\draw[transform canvas={xshift=1.5pt}] (m-1-18) --  (m-2-18);

\draw[transform canvas={xshift=-1.5pt}] (m-1-19) -- (m-2-19);
\draw[transform canvas={xshift=1.5pt}] (m-1-19) --  (m-2-19);

\draw[transform canvas={xshift=-1.5pt}] (m-1-20) -- (m-2-20);
\draw[transform canvas={xshift=1.5pt}] (m-1-20) --  (m-2-20);
\end{tikzpicture}
$};
\end{tikzpicture}
\]
We are now in case \eqref{case:2} of Theorem~\ref{thm:graph-band-to-band}. In this case, $\sigma = \tau (d \inv{c})^3$. Therefore, after setting $\theta = d \inv{c}$, we obtain $\M_{\f} \cong \B_{\theta,1} \oplus \B_{\theta,\omega} \oplus \B_{\theta, \omega^2}$, where $\omega$ is a primitive cube root of unity. In particular, the mapping cone of $\f$ has three indecomposable summands.

Similarly, for any $k, n \geq 1$, if $\sigma = (d \inv{c})^{n+k} \inv{a} b$ and  $\tau = (d \inv{c})^n \inv{a} b$ then $\sigma = \tau (d \inv{c})^k$. Therefore, after setting $\theta = d \inv{c}$, we obtain $\M_{\f} \cong \B_{\theta,1} \oplus \B_{\theta,\omega} \oplus \cdots \oplus \B_{\theta, \omega^{k-1}}$, where $\omega$ is a primitive $k^{\rm th}$ root of unity. This shows  that the mapping cone between two (indecomposable) band complexes can be a direct sum of arbitrarily many indecomposable band complexes.    
\end{example}

\section{Mapping cones of quasi-graph maps involving a band complex} \label{sec:quasi}

The next proposition extends \cite[Prop. 5.2(3) \& (4)]{CPS1} to the case involving a homotopy band and a homotopy string in which the overlap determining the quasi-graph map is longer than the homotopy band. The effect on the mapping cone calculus is `to add a copy of the homotopy band to the homotopy string' to give the homotopy string for the mapping cone. For clarity, we also restate the statement from \cite{CPS1} in which the overlap is shorter than the band, with notation revised to enable effective comparison with the other case. The proof of \cite[Prop. 5.2]{CPS1} holds without change. 

\begin{proposition} \label{prop:quasi-band-and-string}
Let $\sigma$ and $\tau$ be homotopy strings or bands and suppose $\phi \colon \Q_\sigma \rightsquigarrow \Sigma^{-1} \Q_\tau$ is a quasi-graph map determined by a maximal common homotopy substring $\rho$.  Assume further that $\sigma$ and $\tau$ are compatibly oriented for $\phi$ (see \cite[Def. 5.1]{CPS1}). Suppose $\f \colon \Q_\sigma \to \Q_\tau$ is a representative of the homotopy set determined by $\phi$. 
\begin{enumerate}
\item Suppose that  $(\sigma, \lambda)$ is a homotopy band  and $\tau$ is a homotopy string. 
\begin{enumerate}
\item {\rm (\cite[Prop. 5.2(3)]{CPS1})} Suppose $\rho$ is a proper subword of $\sigma$. Then, after suitable rotation, there is a decomposition $\sigma = \rho \alpha$ and a decomposition $\tau = \delta \rho \gamma$. Then $\M_{\f} \cong \P_c$, where $c = \delta \rho \alpha \rho \gamma$.
\item Suppose there is an integer $m \geq 1$ such that $\rho$ is a proper subword of $\sigma^{m+1}$ and $\sigma^m$ is a (not necessarily proper) subword of $\rho$. Then, after suitable rotation of $\sigma$, there is a decomposition $\sigma = \beta \alpha$ such that $\tau = \delta \sigma^m \beta \gamma$, i.e. $\rho = \sigma^m \beta$.
Then $\M_{\f} \cong \P_c$, where $c = \delta \sigma^{m+1} \beta \gamma$.
\end{enumerate}

\item Suppose that $\sigma $ is a homotopy string and $(\tau, \mu)$ is a homotopy band. 
\begin{enumerate}
\item  {\rm (\cite[Prop. 5.2(4)]{CPS1})} Suppose $\rho$ is a proper subword of $\tau$. Then, after suitable rotation, there is a decomposition $\tau = \rho \gamma$ and a decomposition $\sigma = \beta \rho \alpha$. Then $\M_{\f} \cong \P_c$, where $c = \beta \rho \gamma \rho \alpha$.
\item Suppose there is an integer $n \geq 1$ such that $\rho$ is a proper subword of $\tau^{n+1}$ and $\tau^n$ is a (not necessarily proper) subword of $\rho$. Then, after suitable rotation of $\tau$, there is a decomposition $\tau = \delta \gamma$ such that $\sigma = \beta \tau^n \delta \alpha$, i.e. $\rho = \tau^n \delta$.
Then $\M_{\f} \cong \P_c$, where $c = \beta \tau^{n+1} \delta \alpha$.
\end{enumerate}
\end{enumerate}
\end{proposition}

The next theorem extends \cite[Prop. 5.2(2)]{CPS1} to the case involving two homotopy bands and a quasi-graph map determined by an overlap longer than (at least one of) the homotopy bands and corrects the statement in the case when the overlap is shorter than both homotopy bands. 
In both cases the conclusion is the same. Intuitively, this makes sense: the mapping cone associated with a quasi-graph map is actually the mapping cone of one of the maps in the homotopy equivalence class defined by the quasi-graph map. In particular, there is no means by which more than one copy of each band can appear in the word defining the mapping cone.

\begin{theorem} \label{thm:quasi-band-to-band}
Let $(\sigma,\lambda)$ and $(\tau,\mu)$ be homotopy bands and suppose $\phi \colon \B_{\sigma,\lambda} \rightsquigarrow \Sigma^{-1} \B_{\tau,\mu}$ is a quasi-graph map determined by a maximal common homotopy substring $\rho$.  Assume further that $\sigma$ and $\tau$ are compatibly oriented for $\phi$ (see \cite[Def. 5.1]{CPS1}). Suppose $\f \colon \B_{\sigma,\lambda} \to \B_{\tau,\mu}$ is a representative of the homotopy set determined by $\phi$. 
\begin{enumerate}[label=(\arabic*)]
\item If $\rho = \sigma = \tau$, $\lambda = \mu$ and $\phi \colon \B_{\sigma,\lambda} \rightsquigarrow \B_{\sigma,\lambda}$ (that is, $\f \colon \B_{\sigma,\lambda} \to \Sigma \B_{\sigma,\lambda}$), then $\M_{\f} \cong \Sigma \B_{\sigma,\lambda,2}$, where $\B_{\sigma,\lambda,2}$ is the $2$-dimensional band complex (see \cite[\S 5]{ALP}) and
\[
\B_{\sigma,\lambda} \too \B_{\sigma,\lambda,2} \too \B_{\sigma,\lambda} \rightlabel{\f} \Sigma \B_{\sigma,\lambda}
\]
is the Auslander--Reiten triangle starting and ending at $\B_{\sigma,\lambda}$.
\item \label{quasi:generic} Otherwise, we have either
\begin{enumerate}[label=(\alph*)]
\item $\sigma$ is a (not necessarily proper) subword of $\rho$, i.e. after suitable rotation there is a decomposition $\sigma = \beta \alpha$ and an integer $m$ such that $\rho = \sigma^m \beta$; or,
\item $\rho$ is a proper subword of  $\sigma$, i.e. after suitable rotation there is a decomposition $\sigma = \rho\alpha$,  
\end{enumerate}
and, either,
\begin{enumerate}[label=(\alph*),resume]
\item $\tau$ is a (not necessarily proper) subword of $\rho$, i.e. after suitable rotation there is a decomposition $\tau = \delta \gamma$ and an integer $n$ such that $\rho = \tau^n \delta$; or,
\item $\rho$ is a proper subword of $\tau$, i.e. after suitable rotation there is a decomposition $\tau = \rho\gamma$.  
\end{enumerate}
Then there is a homotopy band $\theta$ and an integer $k \geq 1$ such that 
\[
\theta^k =
\begin{cases}
\delta \gamma \beta \alpha & \text{ if (a) \& (c);} \\
\rho \gamma \beta \alpha & \text{ if (a) \& (d);} \\
\delta \gamma \rho \alpha & \text{ if (b) \& (c);} \\
\rho \gamma \rho \alpha & \text{ if (b) \& (d),}
\end{cases}
\quad \text{and} \quad
\M_{\f} \cong \bigoplus_{i=1}^k \B_{\theta,\omega^i \sqrt[k]{- \lambda \mu^{-1}}},
\]
where $\omega$ is a primitive $k^{\it th}$ root of unity.
\end{enumerate}
\end{theorem} 

\begin{proof}
The proof that $\M_{\f} \cong \B_{\tau\sigma,-\lambda\mu^{-1}}$, where $\B_{\tau\sigma,-\lambda\mu^{-1}}$ is the complex induced by the unfolded diagram of the concatenation of the two bands $\tau\sigma$, proceeds exactly as in the proof of \cite[Prop. 5.2]{CPS1}. The argument given in {\it loc. cit.} also shows that $\sigma\tau$ is a (power of a) homotopy band, but did not rule out the case that it is a proper power, i.e. that there may be an integer $k>1$ such that $\tau \sigma = \theta^k$. In this case the decomposition  $\B_{\tau\sigma,-\lambda\mu^{-1}} \cong \bigoplus_{i=1}^k \B_{\theta,\omega^i \sqrt[k]{- \lambda \mu^{-1}}}$ is given by Lemma~\ref{lem:power-of-band}, completing the argument in \cite[Prop. 5.2]{CPS1}.
\end{proof}

\begin{remark}
In Theorem~\ref{thm:quasi-band-to-band}, each of the words defining $\theta^k$ is, after suitable rotation of $\sigma$ and $\tau$ just the concatenation of the two homotopy bands, $\tau \sigma$. However, different possibilities for $\theta$ arise from the precise decompositions of $\sigma$ and $\tau$: for different $\rho$, concatenations $\tau \sigma$ with respect to different decompositions need not be equivalent up to inverting the word or cyclic permutation.
\end{remark}

\begin{example}
Let $\sigma$ and $\tau$ be the homotopy bands given in Example~\ref{ex:two-kroneckers} and $\phi \colon \B_{\sigma,1} \rightsquigarrow \B_{\tau,1}$ be defined by the unfolded diagram \eqref{unfolded}. Let $\f \colon \B_{\sigma,1} \to \Sigma \B_{\tau,1}$ be one of the maps defined by the homotopy class determined by $\phi$. In this case, $\rho = \sigma d \inv{c}$ so that we have the decompositions $\sigma = \beta \alpha$ with $\beta = d \inv{c}$ and $\alpha = \inv{a} b (d \inv{c})^2 \inv{a} b$. Similarly, $\rho = \tau^2$ so that we have the decomposition $\tau = \delta \gamma$ with $\delta = \emptyset$ and $\gamma = \tau = d \inv{c}\, \inv{a} b d \inv{c}$. This puts us in cases (a) and (c) of statement \ref{quasi:generic} of Theorem~\ref{thm:quasi-band-to-band}. Hence, there is an integer $k$ and a homotopy band $\theta$ such that
\[
\theta^k = \delta \gamma \beta \alpha 
= d \inv{c} \, \inv{a} b (d \inv{c})^2 \inv{a} b (d \inv{c})^2 \inv{a} b.
\]
In particular, it follows that in this case $k = 1$. Hence, $\M_{\f} \cong \B_{\theta, -1}$, and the mapping cone of $\f$ has only one indecomposable summand.
\end{example}

\section{Mapping cones of single and double maps involving two band complexes} \label{sec:singleton}

In \cite[Prop. 3.4 \& Prop. 4.2]{CPS1} the possibility that the word defining the mapping cone was a nontrivial homotopy band was not considered. The propostions below complete and correct those statements. The proofs are the same as in \cite{CPS1} up to applying Lemma~\ref{lem:power-of-band} in the case that the resulting word is a nontrivial power of a band.

\begin{proposition} \label{prop:bands-single}
Suppose $(\sigma,\lambda)$ and $(\tau,\mu)$ are homotopy bands and $\f \colon \B_{\sigma,\lambda} \to \B_{\tau,\mu}$ is a single map with single component $f$. Suppose that $\sigma = \beta \sigma_L \sigma_R \alpha$ and $\tau = \delta \tau_L \tau_R \gamma$ are compatibly oriented (in the sense of \cite[Def. 3.1]{CPS1}) for $\f$. 
Then, there exists a homotopy band $\theta$ and an integer $k \geq 1$ such that 
$\theta^k = \beta \sigma_L f \inv{\tau_L} \inv{\delta} \inv{\gamma} \, \inv{\tau_R} \inv{f} \sigma_R \alpha$ 
and
\[
\M_{\f} \cong \bigoplus_{i=1}^k \B_{\theta, \omega^i \sqrt[k]{-\lambda\mu^{-1}}},
\] 
where $\omega$ is a primitive $k^{\it th}$ root of unity.
\end{proposition}

\begin{proposition} \label{prop:bands-double}
Suppose $(\sigma,\lambda)$ and $(\tau,\mu)$ are homotopy bands and $\f \colon \B_{\sigma,\lambda} \to \B_{\tau,\mu}$ is a double map with components $(f_L,f_R)$. Decompose $\sigma = \beta \sigma_L \sigma_C \sigma_R \alpha$ and $\tau = \delta \tau_L \tau_C \tau_R \gamma$ with respect to $\f$. 
Then, there exists a homotopy band $\theta$ and an integer $k \geq 1$ such that 
$\theta^k = \beta \sigma_L f_L \inv{\tau_L} \inv{\delta} \inv{\gamma} \, \inv{\tau_R} \inv{f_R} \sigma_R \alpha$
and 
\[
\M_{\f} \cong \bigoplus_{i=1}^k \B_{\theta, \omega^i \sqrt[k]{-\lambda\mu^{-1}}},
\] 
where $\omega$ is a primitive $k^{\it th}$ root of unity.
\end{proposition}

Finally, we conclude with a specific example showing that it is possible to have a singleton single map between two band complexes whose mapping cone is a `power of a band'.

\begin{example}
Let $\kk = \bC$ and $\Lambda$ be the $\bC$-algebra given by the following quiver with relations.
\[
\begin{tikzpicture}[scale=1.2]
\node (A3) at (0,2){$1$};
\node (B2) at (1,1){$3$};
\node (B3) at (1,2){$2$};
\node (C3) at (2,2){$4$};
\path[->,font=\scriptsize,>=angle 90]
(B2)edge node[right]{$b$}(B3)
(A3)edge node[above]{$a$}(B3)
(B3)edge node[above]{$d$}(C3)
(A3)edge node[below]{$c$}(B2)
(B2)edge node[below]{$e$}(C3)
;
\draw[thick,dotted] (1.3,2) arc (0:170:.3cm) 
(.7,1.2) arc (170:385:.35cm) 
;
\end{tikzpicture}
\]
Let $\sigma = \inv{a} bc \inv{a} b \inv{e} dbc$ and $\tau = e \inv{b} \, \inv{d}$. Consider the singleton single map $\f \colon \B_{\sigma, 36} \to \B_{\tau,4}$ given by the unfolded diagram below,
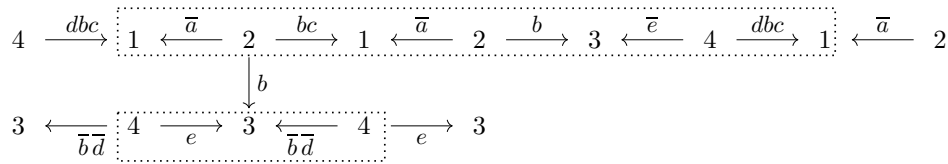
\[
\begin{tikzpicture}
\node[scale=.9] at (0,0){$
\begin{tikzpicture}
\matrix (m) [matrix of math nodes,row sep=2em,column sep=2.4em,minimum width=2em]
  {
  4 & 1 & 2 & 1 & 2 & 3 & 4 & 1 & 2 \\
   3 & 4 & 3 & 4 & 3 &   &   &   & \\};
        
\draw[->] (m-1-1)-- node [pos=.2,anchor=south west,scale=.9]{$dbc$} (m-1-2);
\draw[<-] (m-2-1)-- node [pos=.4,anchor=north west,scale=.9]{$\diaginv{b} \, \diaginv{d}$} (m-2-2);     

\draw[<-]  (m-1-2)-- node [above,scale=.9]{$\diaginv{a}$} (m-1-3);
\draw[->]  (m-2-2)-- node [below,scale=.9]{$e$} (m-2-3);

\draw[->] (m-1-3)-- node [above,scale=.9]{$bc$} (m-1-4);
\draw[<-] (m-2-3)-- node [pos=.4,below,scale=.9]{$\diaginv{b} \, \diaginv{d}$} (m-2-4);

\draw[<-] (m-1-4)-- node [above,scale=.9]{$\diaginv{a}$} (m-1-5);
\draw[->] (m-2-4)-- node [below,scale=.9]{$e$} (m-2-5);

\draw[->] (m-1-5)-- node [above,scale=.9]{$b$} (m-1-6);

\draw[<-] (m-1-6)-- node [above,scale=.9]{$\diaginv{e}$} (m-1-7);

\draw[->] (m-1-7)-- node [above,scale=.9]{$dbc$} (m-1-8);

\draw[<-] (m-1-8)-- node [above,scale=.9]{$\diaginv{a}$} (m-1-9);

\draw[->] (m-1-3)-- node [right,scale=.9]{$b$} (m-2-3);
%
\node (A) [xshift=2em,left of=m-1-2, yshift=-1.4em,above of=m-1-2] {};
\node (B) [xshift=3em,left of=m-1-2, yshift=-1.4em,above of=m-1-8] {};
\node (A') [yshift=.8em,below of=A] {};
\node (B') [yshift=.8em,below of=B] {};

\node (C) [xshift=2em,left of=m-2-2, yshift=-2.1em,above of=m-2-2] {};
\node (D) [xshift=3.4em,left of=m-2-4, yshift=-2.1em,above of=m-2-4] {};
\node (C') [yshift=.75em,below of=C] {};
\node (D') [yshift=.75em,below of=D] {};
\draw[dotted, line width=.7] (A.center) -- (B.center) -- (B'.center) -- (A'.center) --(A.center);
\draw[dotted, line width=.7] (C.center) -- (D.center) -- (D'.center) -- (C'.center) --(C.center)
;
\end{tikzpicture}
$};
\end{tikzpicture}
\]
where one copy of each band $\sigma$ and $\tau$ is highlighted by the dotted boxes. Then we have
\[
\beta \sigma_L f \inv{\tau_L} \inv{\delta} \inv{\gamma} \, \inv{\tau_R} \inv{f} \sigma_R \alpha 
= \inv{a} b \inv{e} dbc \inv{a} b \inv{e} dbc
= (\inv{a} b \inv{e} dbc)^2.
\]
Thus, setting $\theta = \inv{a} b \inv{e} dbc$, we have 
$\M_{\f} = \B_{\theta,3i} \oplus \B_{\theta, -3i}$.
\end{example}


\end{document}